\documentclass[12pt,a4paper,english,american]{article}

\usepackage{amsmath}
\usepackage{amssymb}
\usepackage{esint}

\newcommand{\lyxaddress}[1]{
\par {\raggedright #1
\vspace{1.4em}
\noindent\par}
}

\usepackage{mathrsfs}
\usepackage{amsthm}

\usepackage{babel}

\usepackage{srcltx} 

\theoremstyle{plain}
\newtheorem{theorem}{Theorem}
  \theoremstyle{definition}
  
  \theoremstyle{remark}
  \newtheorem{remark}[theorem]{Remark}
  \theoremstyle{plain}
  \newtheorem{proposition}[theorem]{Proposition}
  \theoremstyle{plain}
  \newtheorem{lemma}[theorem]{Lemma}
  \theoremstyle{plain}
  \newtheorem{corollary}[theorem]{Corollary}
  \theoremstyle{definition}
  
  \theoremstyle{remark}
  \newtheorem*{remark*}{Remark}
  
  \theoremstyle{definition}

\newtheorem*{question*}{\it{QUESTION}}

\newcommand{\der}{\mathop\mathrm{der}\nolimits}

\newcommand{\spec}{\mathop\mathrm{spec}\nolimits}

\renewcommand{\Re}{\mathop\mathrm{Re}\nolimits}

\newcommand{\supp}{\mathop\mathrm{supp}\nolimits}
\newcommand{\Tr}{\mathop\mathrm{Tr}\nolimits}

\newcommand{\Res}{\mathop\mathrm{Res}\nolimits}

\begin{document}

\title{Orthogonal polynomials associated with Coulomb wave functions}

\author{F.~\v{S}tampach$^{1}$, P.~\v{S}\v{t}ov\'\i\v{c}ek$^{2}$}

\date{{}}

\maketitle

\lyxaddress{$^{1}$Department of Applied Mathematics, Faculty of Information
Technology, Czech Technical University in~Prague, Kolejn\'\i~2,
160~00 Praha, Czech Republic}

\lyxaddress{$^{2}$Department of Mathematics, Faculty of Nuclear Science, Czech
Technical University in Prague, Trojanova 13, 12000 Praha, Czech Republic}

\begin{abstract}
  \noindent A class of orthogonal polynomials associated with Coulomb
  wave functions is introduced. These polynomials play a role
  analogous to that the Lommel polynomials do in the theory of Bessel
  functions. The measure of orthogonality for this new class is
  described explicitly. In addition, the orthogonality measure problem
  is also discussed on a more general level. Apart of this, various
  identities derived for the new orthogonal polynomials may be viewed
  as generalizations of some formulas known from the theory of Bessel
  functions. A key role in these derivations is played by a Jacobi
  (tridiagonal) matrix $J_{L}$ whose eigenvalues coincide with
  reciprocal values of the zeros of the regular Coulomb wave function
  $F_{L}(\eta,\rho)$. The spectral zeta function corresponding to the
  regular Coulomb wave function or, more precisely, to the respective
  tridiagonal matrix is studied as well.
\end{abstract}
\vskip\baselineskip\noindent

\vskip0.5\baselineskip\noindent\emph{2010 Mathematical Subject
Classification}: 33C47, 33E15, 11B37

\section{Introduction}

In \cite{Ikebe}, Ikebe showed the zeros of the regular Coulomb wave
function $F_{L}(\eta,\rho)$ and its derivative
$\partial_{\rho}F_{L}(\eta,\rho)$ (regarded as functions of $\rho$) to
be related to eigenvalues of certain compact Jacobi matrices (see
\cite[Chp. 14]{AbramowitzStegun} and references therein for basic
information about the Coulomb wave functions). He applied an approach
originally suggested by Grad and Zakraj\v{s}ek for Bessel functions
\cite{GardZakrajsek}. In more detail, reciprocal values of the nonzero
roots of $F_{L}(\eta,\rho)$ coincide with the nonzero eigenvalues of
the Jacobi matrix
\begin{equation}
  J_{L}=\begin{pmatrix}\lambda_{L+1} & w_{L+1}\\
    w_{L+1} & \lambda_{L+2} & w_{L+2}\\
    & w_{L+2} & \lambda_{L+3} & w_{L+3}\\
    &  & \ddots & \ddots & \ddots
  \end{pmatrix}\label{eq:Jacobi_mat_L}
\end{equation}
where
\begin{equation}
  w_{n} = \frac{\sqrt{(n+1)^{2}+\eta^{2}}}{(n+1)\sqrt{(2n+1)(2n+3)}}
  \quad\mbox{ and }\quad\lambda_{n}=-\frac{\eta}{n(n+1)}
  \label{eq:lambda_w_coulomb}
\end{equation}
for $n=L,L+1,L+2,\dots$. Similarly, reciprocal values of the nonzero
roots of $\partial_{\rho}F_{L}(\eta,\rho)$ coincide with the nonzero
eigenvalues of the Jacobi matrix
\begin{equation}
  \tilde{J}_{L}=\begin{pmatrix}\tilde{\lambda}_{L} & \tilde{w}_{L}\\
    \tilde{w}_{L} & \lambda_{L+1} & w_{L+1}\\
    & w_{L+1} & \lambda_{L+2} & w_{L+2}\\
    &  & \ddots & \ddots & \ddots
  \end{pmatrix}\label{eq:Jacobi_mat_L_tilde}
\end{equation}
where
\begin{equation}
  \tilde{w}_{L} = \sqrt{\frac{2L+1}{L+1}}\, w_{L}
  \quad\mbox{ and }\quad\tilde{\lambda}_{L}
  = -\frac{\eta}{(L+1)^{2}}.
  \label{eq:lambda_w_coulomb_tilde}
\end{equation}
The parameters have been chosen so that $L\in\mathbb{Z}_{+}$
(non-negative integers) and $\eta\in\mathbb{R}$. This is, however,
unnecessarily restrictive and one may extend the set of admissible
values of $L$.  Note also that $J_{L}$ and $\tilde{J}_{L}$ are both
compact, even Hilbert-Schmidt operators on $\ell^{2}(\mathbb{N})$.

Ikebe uses this observation for evaluating the zeros of
$F_{L}(\eta,\rho)$ and $\partial_{\rho}F_{L}(\eta,\rho)$ approximately
by computing eigenvalues of the respective finite truncated Jacobi
matrices. In this paper, we are going to work with the Jacobi matrices
$J_{L}$ and $\tilde{J}_{L}$ as well but pursuing a fully different
goal.  We aim to establish a new class of orthogonal polynomials
(shortly OPs) associated with Coulomb wave functions and to analyze
their properties.  Doing so, we intensively use a formalism which has
been introduced in \cite{StampachStovicek11} and further developed in
\cite{StampachStovicek13a}.  The studied polynomials represent a
two-parameter family generalizing the well known Lommel polynomials
associated with Bessel functions.  Let us also note that another
generalization of Lommel polynomials, in a completely different
direction, has been pointed out by Ismail in \cite{Ismail}, see also
\cite{KoelinkSwarttouw,Koelink99}.

Our primary intention in the study of the new class of OPs was to get
the corresponding orthogonality relation. Before approaching this task
we discuss the problem of finding the measure of orthogonality for a
sequence of OPs on a more general level. In particular, we consider
the situation when a sequence of OPs is determined by a three-term
recurrence whose coefficients satisfy certain convergence condition.
Apart of solving the orthogonality measure problem, various identities
are derived for the new class of OPs which may be viewed as
generalizations of some formulas well known from the theory of Bessel
functions. Finally, the last section is devoted to a study of the
spectral zeta functions corresponding to the regular Coulomb wave
functions or, more precisely, to the respective tridiagonal matrices.
In particular, we derive recursive formulas for values of the zeta
functions. Let us remark that this result can be used to localize the
smallest in modulus zero of $F_{L}(\eta,\rho)$, and hence the spectral
radius of the Jacobi matrix $J_{L}$.

\section{Preliminaries and some useful identities}

\subsection{The function $\mathfrak{F}$}

To have the paper self-contained we first briefly summarize some
information concerning the formalism originally introduced in
\cite{StampachStovicek11} and \cite{StampachStovicek13a} which will be
needed further. Our approach is based on employing a function
$\mathfrak{F}$ defined on the space of complex sequences. By
definition, $\mathfrak{F}:D\rightarrow\mathbb{C}$,
\[
\mathfrak{F}(x)=1+\sum_{m=1}^{\infty}(-1)^{m}
\sum_{k_{1}=1}^{\infty}\,\sum_{k_{2}
  =k_{1}+2}^{\infty}\,\dots\,\sum_{k_{m}=k_{m-1}+2}^{\infty}\, 
x_{k_{1}}x_{k_{1}+1}x_{k_{2}}x_{k_{2}+1}\dots x_{k_{m}}x_{k_{m}+1},
\]
where
\begin{equation}
  D = \left\{ \{x_{k}\}_{k=1}^{\infty}\subset\mathbb{C};\,
    \sum_{k=1}^{\infty}|x_{k}x_{k+1}|<\infty\right\} \!.
  \label{eq:def_D}
\end{equation}
For $x\in D$ one has the estimate
\begin{equation}
  \left|\mathfrak{F}(x)\right|
  \leq \exp\!\left(\sum_{k=1}^{\infty}|x_{k}x_{k+1}|\right)\!.
  \label{eq:F_ineq_exp}
\end{equation}
We identify $\mathfrak{F}(x_{1},x_{2},\dots,x_{n})$ with
$\mathfrak{F}(x)$ where $x=(x_{1},x_{2},\dots,x_{n},0,0,0,\dots)$, and
put $\mathfrak{F}(\emptyset)=1$ where $\emptyset$ stands for the empty
sequence.

Further we list from \cite{StampachStovicek11,StampachStovicek13a}
several useful properties of $\mathfrak{F}$. First,
\begin{equation}
  \mathfrak{F}(x) = \mathfrak{F}(x_{1},\dots,x_{k})\,
  \mathfrak{F}(T^{k}x)-\mathfrak{F}(x_{1},\dots,x_{k-1})x_{k}x_{k+1}\,
  \mathfrak{F}(T^{k+1}x),\quad k=1,2,\dots,
  \label{eq:F_T_recur_k}
\end{equation}
where $x\in D$ and $T$ denotes the shift operator from the left, i.e.
$(Tx)_{k}=x_{k+1}$. In particular, for $k=1$ one gets the rule
\begin{equation}
  \mathfrak{F}(x) = \mathfrak{F}(Tx)-x_{1}x_{2}\,\mathfrak{F}(T^{2}x).
  \label{eq:F_T_recur}
\end{equation}
Second, for $x\in D$ one has
\begin{equation}
  \lim_{n\rightarrow\infty}\mathfrak{F}(T^{n}x) = 1,
  \ \ \lim_{n\rightarrow\infty}\mathfrak{F}(x_{1},x_{2},\dots,x_{n})
  = \mathfrak{F}(x).
  \label{eq:lim_F_T_n}
\end{equation}
Third, one has (see \cite[Subsection 2.3]{StampachStovicek13a})
\begin{eqnarray}
  &  & \mathfrak{F}(x_{1},x_{2},\dots,x_{d})
  \mathfrak{F}(x_{2},x_{3},\dots,x_{d+s})
  -\mathfrak{F}(x_{1},x_{2},\dots,x_{d+s})
  \mathfrak{F}(x_{2},x_{3},\dots,x_{d})\nonumber \\
  &  & =\left(\prod_{j=1}^{d}x_{j}x_{j+1}\right)
  \mathfrak{F}(x_{d+2},x_{d+3},\dots,x_{d+s})
  \label{eq:F_wronsk}
\end{eqnarray}
where $d,s\in\mathbb{Z}_{+}$. By sending $s\rightarrow\infty$ in
(\ref{eq:F_wronsk}) one arrives at the equality
\begin{equation}
  \mathfrak{F}(x_{1},\dots,x_{d})\,\mathfrak{F}(Tx)
  - \mathfrak{F}(x_{2},\dots,x_{d})\,\mathfrak{F}(x)
  = \left(\prod_{k=1}^{d}x_{k}x_{k+1}\right)\mathfrak{F}(T^{d+1}x)
  \label{eq:lincomb_Fx_FTx}
\end{equation}
which is true for any $d\in\mathbb{Z}_{+}$ and $x\in D$.

\subsection{The characteristic function and the Weyl m-function}

Let us consider a semi-infinite symmetric Jacobi matrix $J$ of the
form
\begin{equation}
  J = \begin{pmatrix}\lambda_{0} & w_{0}\\
    w_{0} & \lambda_{1} & w_{1}\\
    & w_{1} & \lambda_{2} & w_{2}\\
    &  & \ddots & \ddots & \ddots
  \end{pmatrix}
  \label{eq:Jacobi_J}
\end{equation}
where $w=\{w_{n}\}_{n=0}^{\infty}\subset(0,+\infty)$ and
$\lambda=\{\lambda_{n}\}_{n=0}^{\infty}\subset\mathbb{R}$. In the
present paper, such a matrix $J$ is always supposed to represent a
unique self-adjoint operator on $\ell^{2}(\mathbb{Z}_{+})$, i.e. there
exists exactly one self-adjoint operator such that the canonical basis
is contained in its domain and its matrix in the canonical basis
coincides with $J$. For example, this hypothesis is evidently
fulfilled if the sequence $\{w_{n}\}$ is bounded. With some abuse of
notation we use the same symbol, $J$, to denote this unique
self-adjoint operator.

In \cite{StampachStovicek13a} we have introduced the characteristic
function $\mathcal{F}_{J}$ for a Jacobi matrix $J$ provided its
elements satisfy the condition
\begin{equation}
  \sum_{n=0}^{\infty}\frac{w_{n}^{\,2}}{|(\lambda_{n}-z)(\lambda_{n+1}-z)|}
  < \infty
  \label{eq:assum_sum_w}
\end{equation}
for some (and hence any) $z\in\mathbb{C}\setminus\der(\lambda)$ where
$\der(\lambda)$ denotes the set of all finite cluster points of the
diagonal sequence $\lambda$, i.e. the set of limit values of all
possible convergent subsequences of $\lambda$. By Corollary~17 in
\cite{StampachStovicek13a}, the condition (\ref{eq:assum_sum_w}) also
guarantees that the matrix $J$ represents a unique self-adjoint
operator on $\ell^{2}(\mathbb{Z}_{+})$. The definition of the
characteristic function reads
\begin{equation}
  \mathcal{F}_{J}(z)
  := \mathfrak{F}\!\left(\left\{ \frac{\gamma_{n}^{\,2}}
      {\lambda_{n}-z}\right\} _{n=0}^{\infty}\right)
  \label{eq:def_FJ_symm}
\end{equation}
where $\{\gamma_{n}\}_{n=0}^{\infty}$ is determined by the
off-diagonal sequence $w$ recursively as follows: $\gamma_{0}=1$ and
$\gamma_{k+1}=w_{k}/\gamma_{k}$, for $k\in\mathbb{Z}_{+}$. The zeros
of the characteristic function have actually been shown in
\cite{StampachStovicek13a} to coincide with the eigenvalues of $J$.
More precisely, under assumption (\ref{eq:assum_sum_w}) it holds true
that
\begin{equation}
  \spec(J)\setminus\der(\lambda)=\spec_{p}(J)\setminus\der(\lambda)
  = \mathfrak{Z}(J)
  \label{eq:spec_p_Z}
\end{equation}
where
\begin{equation}
  \mathfrak{Z}(J) := \left\{ z\in\mathbb{C}\setminus\der(\lambda);
    \,\lim_{u\to z}\,(u-z)^{r(z)}\mathcal{F}_{J}(u)=0\right\}
  \label{eq:def_Z}
\end{equation}
and $r(z):=\sum_{k=0}^{\infty}\delta_{z,\lambda_{k}}\in\mathbb{Z}_{+}$
is the number of occurrences of an element $z$ in the sequence
$\lambda$.  Moreover, the eigenvalues of $J$ have no accumulation
points in $\mathbb{C}\setminus\der(\lambda)$ and all of them are
simple.

Finally, denoting by $\{e_{n};\ n\in\mathbb{Z}_{+}\}$ the canonical
basis in $\ell^{2}(\mathbb{Z}_{+})$, let us recall that the Weyl
m-function $m(z):=\langle e_{0},(J-z)^{-1}e_{0}\rangle$ is expressible
in terms of $\mathfrak{F}$,
\begin{equation}
  m(z) = \frac{1}{\lambda_{0}-z}\,
  \mathfrak{F}\!\left(\left\{ \frac{\gamma_{k}^{\,2}}
      {\lambda_{k}-z}\right\} _{k=1}^{\infty}\right)
  \mathfrak{F}\!\left(\left\{ \frac{\gamma_{k}^{\,2}}
      {\lambda_{k}-z}\right\} _{k=0}^{\infty}\right)^{\!-1}
  \label{eq:Weyl_m}
\end{equation}
for $z\notin\spec(J)\cup\der(\lambda)$. From its definition it is
clear that $m(z)$ is meromorphic on $\mathbb{C}\setminus\der(\lambda)$
with only simple real poles, and the set of these poles coincides with
$\mathfrak{Z}(J)$.

\section{Some general results about orthogonal polynomials \label{sec:OPs}}

The theory of OPs is now developed to a considerable depth. Let us
just mention the basic monographs \cite{Akhiezer,Chihara}. If
convenient, a sequence of OPs, $\{P_{n}\}_{n=0}^{\infty}$, where $\deg
P_{n}=n$, may be supposed to be already normalized. Then one way how
to define such a sequence is by requiring the orthogonality relation
\begin{equation}
  \int_{\mathbb{R}}P_{m}(x)P_{n}(x)\,\mbox{d}\mu(x)
  = \delta_{mn},\quad m,n\in\mathbb{Z}_{+},
  \label{eq:OGrel_OPs}
\end{equation}
with respect to a positive Borel measure $\mu$ on $\mathbb{R}$ such
that
\[
\int_{\mathbb{R}}x^{2n}\,\mbox{d}\mu(x)
< \infty,\ \ \forall n\in\mathbb{Z}_{+}.
\]
Without loss of generality one may assume that $\mu$ is a probability
measure, i.e. $\mu(\mathbb{R})=1$, and $P_{0}(x)=1$. As usual, $\mu$
is unambiguously determined by the distribution function
$x\mapsto\mu((-\infty,x])$.  In particular, the distribution function
is supposed to be continuous from the right. With some abuse of
notation, the distribution function will again be denoted by the
symbol $\mu$. The set of monomials, $\{x^{n};\ n\in\mathbb{Z}_{+}\}$,
is required to be linearly independent in
$L^{2}(\mathbb{R},\mbox{d}\mu)$ and so the function $\mu$ should have
an infinite number of points of increase.

It is well known that a sequence of OPs, if normalized, satisfies a
three-term recurrence relation,
\begin{equation}
  xP_{n}(x) = w_{n-1}P_{n-1}(x)+\lambda_{n}P_{n}(x)+w_{n}P_{n+1}(x),
  \quad n\in\mathbb{N},
  \label{eq:recur_OPs}
\end{equation}
with the initial conditions $P_{0}(x)=1$ and
$P_{1}(x)=(x-\lambda_{0})/w_{0}$, where
$\{\lambda_{n}\}_{n=0}^{\infty}$ is a real sequence and
$\{w_{n}\}_{n=0}^{\infty}$ is a positive sequence
\cite{Akhiezer,Chihara}. However, due to Favard's theorem, the
opposite statement is also true. For any sequence of real polynomials,
$\{P_{n}\}_{n=0}^{\infty}$, with $\deg P_{n}=n$, satisfying the
recurrence (\ref{eq:recur_OPs}) with the above given initial
conditions there exists a unique positive functional on the space of
real polynomials making this sequence orthonormal. Moreover, if the
matrix $J$ given in (\ref{eq:Jacobi_J}) represents a unique
self-adjoint operator on $\ell^{2}(\mathbb{Z}_{+})$ then this
functional is induced by a unique positive Borel measure $\mu$ on
$\mathbb{R}$.  This means that (\ref{eq:OGrel_OPs}) is fulfilled. In
other words, in that case Hamburger's moment problem is determinate;
see, for instance, \S4.1.1 and Corollary~2.2.4 in \cite{Akhiezer} or
Theorem~3.4.5 in \cite{EKoelink}.

Using (\ref{eq:F_T_recur}) one easily verifies that the solution of
(\ref{eq:recur_OPs}) with the given initial conditions is related to
$\mathfrak{F}$ through the identity
\begin{equation}
  P_{n}(x) = \prod_{k=0}^{n-1}\left(\frac{x-\lambda_{k}}{w_{k}}\right)
  \mathfrak{F}\!\left(\left\{ \frac{\gamma_{k}^{\,2}}
      {\lambda_{k}-x}\right\} _{k=0}^{n-1}\right),
  \quad n\in\mathbb{Z}_{+}.
  \label{eq:OPs_P_F}
\end{equation}
A second linearly independent solution of (\ref{eq:recur_OPs}) can be
written in the form
\begin{equation}
  Q_{n}(x) = \frac{1}{w_{0}}\prod_{k=1}^{n-1}
  \left(\frac{x-\lambda_{k}}{w_{k}}\right)
  \mathfrak{F}\!\left(\left\{ \frac{\gamma_{k+1}^{\,2}}
      {\lambda_{k+1}-x}\right\} _{k=0}^{n-2}\right),
  \quad n\in\mathbb{N}.
  \label{eq:OPs_Q_F}
\end{equation}
The latter solution satisfies the initial conditions $Q_{0}(x)=0$ and
$Q_{1}(x)=1/w_{0}$.

Being given a sequence of OPs, $\{P_{n}\}_{n=0}^{\infty}$, defined via
the recurrence rule (\ref{eq:recur_OPs}), i.e. via formula
(\ref{eq:OPs_P_F}), a crucial question is how the measure of
orthogonality looks like.  Relying on the function $\mathfrak{F}$ we
provide a partial description of the measure $\mu$. Doing so we
confine ourselves to Jacobi matrices for which the set of cluster
points of the diagonal sequence $\lambda$ is discrete. This assumption
is not too restrictive, though, since it turns out that
$\der(\lambda)$ is a one-point set or even empty in many practical
applications of interest.

\begin{proposition}\label{thm:OG_relation}
  Let $J$ be a Jacobi matrix introduced in (\ref{eq:Jacobi_J}) and
  $\der(\lambda)$ be composed of isolated points only. Suppose there
  exists $z_{0}\in\mathbb{C}$ such that (\ref{eq:assum_sum_w}) is
  fulfilled for $z=z_{0}$. Then the orthogonality relation for the
  sequence of OPs determined in (\ref{eq:recur_OPs}) reads
  \begin{equation}
    \int_{\mathbb{R}}P_{m}(x)P_{n}(x)\,
    \mbox{d}\nu(x)+\sum_{x\in\mathcal{D}}
    \frac{P_{m}(x)P_{n}(x)}{\|P(x)\|^{2}}
    = \delta_{mn},\ \ m,n\in\mathbb{Z}_{+},
    \label{eq:gener_OG_rel}
\end{equation}
where $\mathcal{D}=\spec_{p}(J)\cap\der(\lambda)$ and $\|P(x)\|$
stands for the $\ell^{2}$-norm of the vector
$P(x)=(P_{0}(x),P_{1}(x),\dots)$.  The measure $d\nu$ is positive,
purely discrete and supported on the set $\mathfrak{Z}(J)$. The
magnitude of jumps of the step function $\nu(x)$ at those points
$x\in\mathfrak{Z}(J)$ which do not belong to the range of $\lambda$
equals
\begin{equation}
  \nu(x)-\nu(x-0)=\frac{1}{x-\lambda_{0}}\,
  \mathfrak{F}\!\left(\left\{ \frac{\gamma_{k+1}^{\,2}}
      {\lambda_{k+1}-x}\right\} _{k=0}^{\infty}\right)
  \left[\frac{\mbox{d}}{\mbox{\mbox{d}}x}\,
    \mathfrak{F}\!\left(\left\{ \frac{\gamma_{k}^{\,2}}
        {\lambda_{k}-x}\right\} _{k=0}^{\infty}\right)\right]^{-1}.
  \label{eq:jumps}
\end{equation}
\end{proposition}

\begin{remark*}
  In Proposition~\ref{thm:OG_relation} we avoided considering the
  points from $\mathfrak{Z}(J)$ which belong to the range of
  $\lambda$. We remark, however, that such points, if any, can be
  treated as well, similarly to (\ref{eq:jumps}), though in somewhat
  more complicated way. But we omit the details for the sake of
  simplicity.
\end{remark*}

\begin{proof}
  Let $E_{J}$ stand for the projection-valued spectral measure of the
  self-adjoint operator $J$. As is well known, the measure of
  orthogonality $\mu$ is related to $E_{J}$ by the identity
  \begin{equation}
    \mu(M) = \langle e_{0},E_{J}(M)e_{0}\rangle
    \label{eq:mu_eq_spec_meas}
  \end{equation}
  holding for any Borel set $M\subset\mathbb{R}$. Here again, $e_{0}$
  denotes the first vector of the canonical basis in
  $\ell^{2}(\mathbb{Z}_{+})$.  Moreover, $\supp(\mu)=\spec(J)$. In
  fact, let us recall that (\ref{eq:mu_eq_spec_meas}) follows from the
  observation that $e_{n}=P_{n}(J)e_{0}$ for all $n\in\mathbb{Z}_{+}$
  and from the Spectral Theorem since
  \[
  \delta_{mn} = \langle e_{m},e_{n}\rangle
  = \langle e_{0},P_{m}(J)P_{n}(J)e_{0}\rangle
  = \int_{\mathbb{R}}P_{m}(x)P_{n}(x)\,\mbox{d}\mu(x).
  \]

  The set $\der(\lambda)$ is closed and, by hypothesis, discrete and
  therefore at most countable. Further we know, referring to
  (\ref{eq:spec_p_Z}), that the part of the spectrum of $J$ lying in
  $\mathbb{C}\setminus\der(\lambda)$ is discrete, too. Consequently,
  $\spec(J)$ is countable and therefore the continuous part of the
  spectral measure $E_{J}$ necessarily vanishes, i.e. $J$ has a pure
  point spectrum. In that case, of course, in order to determine the
  spectral measure $E_{J}$ it suffices to determine the projections
  $E_{J}(\{x\})$ for all $x\in\spec_{p}(J)$. Since the vector $P(z)$
  is a formal solution of $(J-z)P(z)=0$, unique up to a constant
  multiplier, one has the well known criterion $x\in\spec_{p}(J)$ iff
  $\|P(x)\|<\infty$. Moreover, $P_{0}(x)=1$ and so
  \[
  \langle e_{0},E_{J}(\{x\})e_{0}\rangle
  = \frac{|\langle P(x),e_{0}\rangle|^{2}}{\|P(x)\|^{2}}
  = \frac{1}{\|P(x)\|^{2}}\,.
  \]

  The point spectrum of $J$ may be split into two disjoint sets,
  $\spec_{p}(J)=\mathfrak{Z}(J)\cup\mathcal{D}$.  The Hilbert space
  and the spectral measure decompose correspondingly.  Put
  \[
  J' = J\, E_{J}(\mathfrak{Z}(J))\ \ \text{and}\ \ \nu(x)
  = \langle e_{0},E_{J'}((-\infty,x])e_{0}\rangle
  \ \ \text{for}\ x\in\mathbb{R}.
  \]
  Then the measure $\mbox{d}\nu$ is supported on $\mathfrak{Z}(J)$ and
  \[
  \int_{\mathbb{R}}f(x)\,\mbox{d}\mu(x)
  = \int_{\mathbb{R}}f(x)\,\mbox{d}\nu(x)+\sum_{x\in\mathcal{D}}
  \frac{f(x)}{\|P(x)\|^{2}}
  \]
  for all $f\in C(\mathbb{R})$. As pointed out in (\ref{eq:spec_p_Z}),
  any $x\in\mathfrak{Z}(J)$ is a simple isolated eigenvalue of $J$.
  As usual, $E_{J}(\{x\})$ can be written as the Riezs spectral
  projection.  Choosing $\epsilon>0$ sufficiently small one has
  \[
  \langle e_{0},E_{J}(\{x\})e_{0}\rangle
  = -\frac{1}{2\pi i}\oint_{|x-z|=\epsilon}m(z)\,\mbox{d}z=-\Res(m,x).
  \]
  If, in addition, $x$ does not belong to the range of $\lambda$ then,
  in view of (\ref{eq:Weyl_m}) and (\ref{eq:def_FJ_symm}),
  (\ref{eq:def_Z}) (with $r(x)=0$), we may evaluate
  \[
  \Res(m,x) = \frac{1}{\lambda_{0}-x}\,
  \mathfrak{F}\!\left(\left\{ \frac{\gamma_{k+1}^{\,2}}
      {\lambda_{k+1}-x}\right\} _{k=0}^{\infty}\right)
  \left[\,\frac{\mbox{d}}{\mbox{d}z}\bigg|_{z=x}
    \mathfrak{F}\!\left(\left\{ \frac{\gamma_{k}^{\,2}}
        {\lambda_{k}-z}\right\} _{k=0}^{\infty}\right)\right]^{-1}\!.
  \]
  This concludes the proof.
\end{proof}

\begin{remark}\label{rem:isol_eigen_only}
  Of course, the sum on the LHS of (\ref{eq:gener_OG_rel}) is void if
  $\der(\lambda)=\emptyset$.  The sum also simplifies in the
  particular case when $J$ is a compact operator satisfying
  (\ref{eq:assum_sum_w}). One can readily see that this happens iff
  $\lambda_{n}\rightarrow0$ and $w\in\ell^{2}(\mathbb{Z}_{+})$. Then
  Proposition~\ref{thm:OG_relation} is applicable and the
  orthogonality relation (\ref{eq:gener_OG_rel}) takes the form
  \[
  \int_{\mathbb{R}}P_{n}(x)P_{m}(x)\,\mbox{d}\nu(x)
  + \Lambda_{0}P_{n}(0)P_{m}(0)=\delta_{mn}.
  \]
  If $J$ is invertible then $\Lambda_{0}$ vanishes but in general
  $\Lambda_{0}$ may be strictly positive as demonstrated, for
  instance, by the example of $q$-Lommel polynomials, see
  \cite[Theorem 4.2]{Koelink_etal}.
\end{remark}

For the intended applications of Proposition~\ref{thm:OG_relation} the
following particular case is of importance. Let
$\lambda\in\ell^{1}(\mathbb{Z}_{+})$ be real and
$w\in\ell^{2}(\mathbb{Z}_{+})$ positive. Then $J$ is compact and
(\ref{eq:assum_sum_w}) holds for any $z\neq0$ not belonging to the
range of $\lambda$. Moreover, the characteristic function of $J$ can
be regularized with the aid of the entire function
\[
\phi_{\lambda}(z):=\prod_{n=0}^{\infty}(1-z\lambda_{n}).
\]
Let us define
\begin{equation}
  \mathcal{G}_{J}(z) := \begin{cases}
    \phi_{\lambda}(z)\mathcal{F}_{J}(z^{-1}) & \mbox{ if }z\neq0,\\
    1 & \mbox{ if }z=0.
  \end{cases}
  \label{eq:def_G_J}
\end{equation}
The function $\mathcal{G}_{J}$ is entire and, referring to
(\ref{eq:spec_p_Z}), one has
\begin{equation}
  \spec(J) = \{0\}\cup\left\{ z^{-1};\,\mathcal{G}_{J}(z)=0\right\} \!. 
  \label{eq:specJ_GJ}
\end{equation}
Since
\[
m(z) = \int_{\mathbb{R}}\frac{\mbox{d}\mu(x)}{x-z}
\]
where $\mbox{d}\mu$ is the measure from (\ref{eq:mu_eq_spec_meas}),
formula (\ref{eq:Weyl_m}) implies that the identity
\begin{equation}
  \int_{\mathbb{R}}\frac{\mbox{d}\mu(x)}{1-xz}
  = \frac{\mathcal{G}_{J^{(1)}}(z)}{\mathcal{G}_{J}(z)}
  \label{eq:int_eq_ratio_G}
\end{equation}
holds for any $z\notin\mathcal{G}_{J}^{-1}(\{0\})$. Here $J^{(1)}$
denotes the Jacobi operator determined by the diagonal sequence
$\{\lambda_{n+1}\}_{n=0}^{\infty}$ and the weight sequence
$\{w_{n+1}\}_{n=0}^{\infty}$, see (\ref{eq:J_sup_k}) below.

Let us denote by $\{\mu_{n}\}_{n=1}^{\infty}$ the set of non-zero
eigenvalues of the compact operator $J$. Remember that all eigenvalues
of $J$ are necessarily simple and particularly the multiplicity of $0$
as an eigenvalues of $J$ does not exceed $1$. Since $\mbox{d}\mu$ is
supported by $\spec(J)$, formula (\ref{eq:int_eq_ratio_G}) yields the
Mittag-Leffler expansion
\begin{equation}
  \Lambda_{0}+\sum_{k=1}^{\infty}\frac{\Lambda_{k}}{1-\mu_{k}z}
  = \frac{\mathcal{G}_{J^{(1)}}(z)}{\mathcal{G}_{J}(z)}
  \label{eq:Mittag-Leffler_ratio_G}
\end{equation}
where $\Lambda_{k}$ denotes the jump of the piece-wise constant
function $\mu(x)$ at $x=\mu_{k}$, and similarly for $\Lambda_{0}$ and
$x=0$. From (\ref{eq:Mittag-Leffler_ratio_G}) one deduces that
\begin{equation}
  \Lambda_{k} = \lim_{z\rightarrow\mu_{k}^{-1}}(1-\mu_{k}z)\,
  \frac{\mathcal{G}_{J^{(1)}}(z)}{\mathcal{G}_{J}(z)}
  = -\mu_{k}\,\frac{\mathcal{G}_{J^{(1)}}(\mu_{k}^{-1})}
  {\mathcal{G}'_{J}(\mu_{k}^{-1})}
\end{equation}
for $k\in\mathbb{N}$. This can be viewed as a regularized version of
the identity (\ref{eq:jumps}) in this particular case. We have shown
the following proposition.

\begin{proposition}\label{prop:OGrel_Gfunc}
  Let $\lambda$ be a real sequence from $\ell^{1}(\mathbb{Z}_{+})$ and
  $w$ be a positive sequence from $\ell^{2}(\mathbb{Z}_{+})$. Then the
  measure of orthogonality $\mbox{d}\mu$ for the corresponding
  sequence of OPs defined in (\ref{eq:recur_OPs}) fulfills
  \[
  \supp(\mbox{d}\mu)\setminus\{0\}=\{z^{-1};\,\mathcal{G}_{J}(z)=0\}
  \]
  where the RHS is a bounded discrete subset of $\mathbb{R}$ with $0$
  as the only accumulation point. Moreover, for
  $x\in\supp(\mbox{d}\mu)\setminus\{0\}$ one has
  \begin{equation}
    \mu(x)-\mu(x-0)=-x\,\frac{\mathcal{G}_{J^{(1)}}(x^{-1})}
    {\mathcal{G}'_{J}(x^{-1})}\,.
    \label{eq:jump_mu_in_x}
  \end{equation}
\end{proposition}

Let us denote $\xi_{-1}(z):=\mathcal{G}_{J}(z)$ and
\begin{equation}
  \xi_{k}(z) := \left(\prod_{l=0}^{k-1}w_{l}\right)z^{k+1}\,
  \mathcal{G}_{J^{(k+1)}}(z),\quad k\in\mathbb{Z}_{+},
  \label{eq:psi_k_def}
\end{equation}
where
\begin{equation}
  J^{(k)} = \begin{pmatrix}\lambda_{k} & w_{k}\\
    w_{k} & \lambda_{k+1} & w_{k+1}\\
    & w_{k+1} & \lambda_{k+2} & w_{k+2}\\
    &  & \ddots & \ddots & \ddots
  \end{pmatrix}\!.
  \label{eq:J_sup_k}
\end{equation}

\begin{lemma}\label{lem:vector_xi}
  Let $\lambda\in\ell^{1}(\mathbb{Z}_{+})$,
  $w\in\ell^{2}(\mathbb{Z}_{+})$ and $z\neq0$. Then the vector
  \begin{equation}
    \xi(z) = (\xi_{0}(z),\xi_{1}(z),\xi_{2}(z),\dots)
    \label{eq:vector_xi}
  \end{equation}
  belongs to $\ell^{2}(\mathbb{Z}_{+})$, and one has
  \begin{equation}
    \frac{1}{z^{2}}\sum_{k=0}^{\infty}\xi_{k}(z)^{2}
    = \xi_{-1}(z)\xi'_{0}(z)-\xi'_{-1}(z)\xi_{0}(z).
    \label{eq:ChristDarb_form}
  \end{equation}
  Moreover, the vector $\xi(z)$ is nonzero and so
  \begin{equation}
    \xi_{-1}(z)\xi'_{0}(z)-\xi'_{-1}(z)\xi_{0}(z)
    > 0,\ \forall z\in\mathbb{R}\setminus\{0\},
    \label{eq:sign_x0_x1}
  \end{equation}
  provided the sequences $\lambda$ and $w$ are both real.
\end{lemma}

\begin{proof}
  First, choose $N\in\mathbb{Z}_{+}$ so that $z^{-1}\neq\lambda_{k}$
  for all $k>N$. This is clearly possible since
  $\lambda_{n}\rightarrow0$ as $n\rightarrow\infty$. Then we have,
  referring to (\ref{eq:F_ineq_exp}) and (\ref{eq:def_G_J}),
  \[
  |\mathcal{G}_{J^{(k)}}(z)|
  \leq \exp\!\left(\sum_{j=N+1}^{\infty}|z||\lambda_{j}|
    + \sum_{j=N+1}^{\infty}\frac{|z|^{2}\,|w_{j}|^{\,2}}
    {|(1-z\lambda_{j})(1-z\lambda_{j+1})|}\right)\ \ \text{for}\ k>N.
  \]
  Observing that $w_{n}\rightarrow0$ as $n\rightarrow\infty$, one
  concludes that there exists a constant $C>0$ such that
  \[
  |z|^{k+1}\,\prod_{l=0}^{k-1}w_{l}\leq C\,2^{-k}\ \ \text{for}\ k>N.
  \]
  These estimates obviously imply the square summability of the vector
  $\xi(z)$.

  Second, with the aid of (\ref{eq:F_T_recur}) one verifies that for
  all $z\neq0$ and $k\in\mathbb{Z}_{+}$,
  \[
  w_{k-1}\xi_{k-1}(z)+(\lambda_{k}-z^{-1})\xi_{k}(z)+w_{k}\xi_{k+1}(z)=0
  \]
  where we put $w_{-1}:=1$. From here one deduces that the equality
  \begin{equation}
    (z^{-1}-x^{-1})\xi_{k}(z)\xi_{k}(x)=W_{k}(x,z)-W_{k-1}(x,z),
    \label{eq:toChristDarb_form}
  \end{equation}
  with
  \[
  W_{k}(x,z)
  = w_{k}\left(\xi_{k+1}(z)\xi_{k}(x)-\xi_{k+1}(x)\xi_{k}(z)\right),
  \]
  holds for all $k\in\mathbb{Z}_{+}$. Now one can derive
  (\ref{eq:ChristDarb_form}) from (\ref{eq:toChristDarb_form}) in a
  routine way.

  Finally, observe that the first equality in (\ref{eq:lim_F_T_n})
  implies the limit
  \[
  \lim_{k\rightarrow\infty}\mathcal{G}_{J^{(k)}}(z) = 1.
  \]
  Referring to (\ref{eq:psi_k_def}) this means $\xi_{k}(z)\neq0$ for
  all sufficiently large $k$. \end{proof}

\begin{proposition}\label{prop:psi_eigenvector}
  Let $\lambda\in\ell^{1}(\mathbb{Z}_{+})$ be real,
  $w\in\ell^{2}(\mathbb{Z}_{+})$ be positive and $z\neq0$. If $z^{-1}$
  is an eigenvalue of the Jacobi operator $J$ given in
  (\ref{eq:Jacobi_J}) then the vector (\ref{eq:vector_xi}) is a
  corresponding eigenvector.
\end{proposition}

\begin{proof}
  As recalled in (\ref{eq:specJ_GJ}), $z^{-1}$ is an eigenvalue of $J$
  iff $\mathcal{G}_{J}(z)\equiv\xi_{-1}(z)=0$. Then one readily
  verifies, with the aid of (\ref{eq:F_T_recur}), that $\xi(z)$ is a
  formal solution of the eigenvalue equation $(J-z^{-1})\xi(z)=0$. By
  Lemma~\ref{lem:vector_xi}, $\xi(z)\neq0$. It is even true that
  $\xi_{0}(z)\neq0$. Indeed, if $\xi_{-1}(z)=\xi_{0}(z)=0$ then, by
  the recurrence, $\xi_{k}(z)=0$ for all $k\in\mathbb{Z}_{+}$, a
  contradiction.  Moreover, Lemma~\ref{lem:vector_xi} also tells us
  that $\xi(z)\in\ell^{2}(\mathbb{Z}_{+})$.
\end{proof}

\begin{proposition}\label{prop:G_J_simple_zeros}
  Let $\lambda\in\ell^{1}(\mathbb{Z}_{+})$ be real and
  $w\in\ell^{2}(\mathbb{Z}_{+})$ be positive. Then the zeros of the
  function $\mathcal{G}_{J}$ are all real and simple, and form a
  countable subset of $\mathbb{R}\setminus\{0\}$ with no finite
  accumulation points. Furthermore, the functions $\mathcal{G}_{J}$
  and $\mathcal{G}_{J^{(1)}}$ have no common zeros, and the zeros of
  the same sign of $\mathcal{G}_{J}$ and $\mathcal{G}_{J^{(1)}}$
  mutually separate each other, i.e. between any two consecutive zeros
  of $\mathcal{G}_{J}$ which have the same sign there is a zero of
  $\mathcal{G}_{J^{(1)}}$ and vice versa. \end{proposition}

\begin{proof}
  The first part of the proposition follows from (\ref{eq:specJ_GJ}).
  In fact, all zeros of $\mathcal{G}_{J}$ are surely real since $J$ is
  a Hermitian operator in $\ell^{2}(\mathbb{Z}_{+})$. Moreover, $J$ is
  compact and all its eigenvalues are simple. Therefore the set of
  reciprocal values of nonzero eigenvalues of $J$ is countable and has
  no finite accumulation points.

  Thus we know that the zeros of $\mathcal{G}_{J}$ and
  $\mathcal{G}_{J^{(1)}}$ are all located in
  $\mathbb{R}\setminus\{0\}$ and $\xi_{-1}(z)=\mathcal{G}_{J}(z)$,
  $\xi_{0}(z)=z\mathcal{G}_{J^{(1)}}(z)$. Hence, as far as the zeros
  are concerned and we are separated from the origin, we can speak
  about $\xi_{-1}$ and $\xi_{0}$ instead of $\mathcal{G}_{J}$ and
  $\mathcal{G}_{J^{(1)}}$, respectively. The remainder of the
  proposition can be deduced from (\ref{eq:sign_x0_x1}) in a usual
  way. Suppose a zero of $\xi_{-1}$, called $z$, is not simple. Then
  $\xi_{-1}(z)=\xi'_{-1}(z)=0$ which leads to a contradiction with
  (\ref{eq:sign_x0_x1}). From (\ref{eq:sign_x0_x1}) it is immediately
  seen, too, that $\xi_{-1}$ and $\xi_{0}$ have no common zeros in
  $\mathbb{R}\setminus\{0\}$. Furthermore, suppose $z_{1}$ and
  {$z_{2}$} are two consecutive zeros of $\xi_{-1}$ of the same sign.
  Since these zeros are simple, the numbers $\xi_{-1}'(z_{1})$ and
  $\xi_{-1}'(z_{2})$ differ in sign.  From (\ref{eq:sign_x0_x1}) one
  deduces that $\xi_{0}(z_{1})$ and $\xi_{0}(z_{2})$ must differ in
  sign as well. Consequently, there is at least one zero of $\xi_{0}$
  lying between $z_{1}$ and $z_{2}$.  An entirely analogous argument
  applies if the roles of $\xi_{-1}$ and $\xi_{0}$ are interchanged.
\end{proof}

\section{Lommel polynomials}

\subsection{Basic properties and the orthogonality relation \label{subsec:Lommel_basic}}

In this section we deal with the Lommel polynomials as one of the
simplest and most interesting examples to demonstrate the general
results derived in Section~\ref{sec:OPs}. This is done with the
perspective of approaching our main goal in this paper, namely a
generalization of the Lommel polynomials established in the next
section. Let us note that although the Lommel polynomials are
expressible in terms of hypergeometric series, they do not fit into
Askey's scheme of hypergeometric orthogonal polynomials
\cite{Koekoek_etal}.

Recall that the Lommel polynomials arise in the theory of Bessel
function (see, for instance, \cite[?? 9.6-9.73]{Watson} or
\cite[Chp.~VII]{Erdelyi_etal_II}).  They can be written explicitly in
the form
\begin{equation}
  R_{n,\nu}(x) = \sum_{k=0}^{[n/2]}(-1)^{k}\binom{n-k}{k}
  \frac{\Gamma(\nu+n-k)}{\Gamma(\nu+k)}
  \left(\frac{2}{x}\right)^{\! n-2k}
  \label{eq:Lommel_explicit}
\end{equation}
where $n\in\mathbb{Z}_{+}$, $\nu\in\mathbb{C}$,
$-\nu\notin\mathbb{Z}_{+}$ and $x\in\mathbb{C}\setminus\{0\}$. Here we
stick to the traditional terminology though, obviously, $R_{n,\nu}(x)$
is a polynomial in the variable $x^{-1}$ rather than in $x$.
Proceeding by induction in $n\in\mathbb{Z}_{+}$ one easily verifies
the identity
\begin{equation}
  R_{n,\nu}(x) = \left(\frac{2}{x}\right)^{\! n}
  \frac{\Gamma(\nu+n)}{\Gamma(\nu)}\,\mathfrak{F}\!
  \left(\left\{ \frac{x}{2(\nu+k)}\right\} _{k=0}^{n-1}\right)\!.
  \label{eq:lommel_F}
\end{equation}
As is well known, the Lommel polynomials are directly related to Bessel
functions,
\begin{eqnarray*}
  R_{n,\nu}(x) & = & \frac{\pi x}{2}
  \left(Y_{-1+\nu}(x)J_{n+\nu}(x)-J_{-1+\nu}(x)Y_{n+\nu}(x)\right)\\
  &  & \frac{\pi x}{2\sin(\pi\nu)}\left(J_{1-\nu}(x)
    J_{n+\nu}(x)+(-1)^{n}J_{-1+\nu}(x)J_{-n-\nu}(x)\right).
\end{eqnarray*}
From here and relation (\ref{eq:Bessel_recur}) below it is seen that
the Lommel polynomials obey the recurrence
\begin{equation}
  R_{n+1,\nu}(x) = \frac{2\,(n+\nu)}{x}\, R_{n,\nu}(x)
  -R_{n-1,\nu}(x),\ \ n\in\mathbb{Z}_{+},
  \label{eq:Lommel_recur}
\end{equation}
with the initial conditions $R_{-1,\nu}(x)=0$, $R_{0,\nu}(x)=1$.

The original meaning of the Lommel polynomials reveals the formula
\begin{equation}
  J_{\nu+n}(x) = R_{n,\nu}(x)J_{\nu}(x)-R_{n-1,\nu+1}(x)
  J_{\nu-1}(x)\ \ \text{for}\ n\in\mathbb{Z}_{+}.
  \label{eq:Lommel_in_lincomb_Bessel}
\end{equation}
As firstly observed by Lommel in 1871,
(\ref{eq:Lommel_in_lincomb_Bessel}) can be obtained by iterating the
basic recurrence relation for the Bessel functions, namely
\begin{equation}
  J_{\nu+1}(x) = \frac{2\nu}{x}\, J_{\nu}(x)-J_{\nu-1}(x).
  \label{eq:Bessel_recur}
\end{equation}
Let us remark that (\ref{eq:Lommel_in_lincomb_Bessel}) immediately
follows from (\ref{eq:lincomb_Fx_FTx}), (\ref{eq:lommel_F}) and
formula (\ref{eq:BesselJ_rel_F}) below.

The orthogonality relation for Lommel polynomials is known explicitly
and is expressed in terms of the zeros of the Bessel function of order
$\nu-1$ as explained, for instance, in
\cite{Dickinson54,Dickinson_etal},
see also \cite[Chp. VI, S 6]{Chihara} and \cite{Ismail}. This relation
can also be rederived as a corollary of
Proposition~\ref{prop:OGrel_Gfunc}.  For $\nu>-1$ and
$n\in\mathbb{Z}_{+}$, set temporarily
\[
\lambda_{n} = 0\quad\mbox{ and }\quad w_{n}
= 1/\sqrt{(\nu+n+1)(\nu+n+2)}\,.
\]
Then the corresponding Jacobi operator $J$ is compact, self-adjoint
and $0$ is not an eigenvalue. In fact, invertibility of $J$ can be
verified straightforwardly by solving the formal eigenvalue equation
for $0$. Referring to (\ref{eq:BesselJ_rel_F}), the regularized
characteristic function of $J$ equals
\[
\mathcal{G}_{J}(z) = \mathcal{F}_{J}(z^{-1})
= \Gamma(\nu+1)\, z^{-\nu}J_{\nu}(2z).
\]
Consequently, the support of the measure of orthogonality turns out to
coincide with the zero set of $J_{\nu}(z)$. Remember that
$x^{-\nu}J_{\nu}(x)$ is an even function. Let $j_{k,\nu}$ stand for
the $k$-th positive zero of $J_{\nu}(x)$ and put
$j_{-k,\nu}=-j_{k,\nu}$ for $k\in\mathbb{N}$.
Proposition~\ref{prop:OGrel_Gfunc} then tells us that the
orthogonality relation takes the form
\[
-2(\nu+1)\sum_{k\in\mathbb{Z}
  \setminus\{0\}}\frac{J_{\nu+1}(j_{k,\nu})}
{j_{k,\nu}^{\,2}\,J_{\nu}'(j_{k,\nu})}\,
P_{m}\!\left(\frac{2}{j_{k,\nu}}\right)
P_{n}\!\left(\frac{2}{j_{k,\nu}}\right) = \delta_{mn}
\]
where $J_{\nu}'(x)$ denotes the partial derivative of $J_{\nu}(x)$
with respect to $x$.

Furthermore, (\ref{eq:OPs_P_F}) and (\ref{eq:lommel_F}) imply
\begin{equation}
  R_{n,\nu+1}(x) = \sqrt{\frac{\nu+1}{\nu+n+1}}\,
  P_{n}\!\left(\frac{2}{x}\right).
  \label{eq:Lommel_R_eq_Pn}
\end{equation}
Using the identity
\[
\partial_{x}J_{\nu}(x) = \frac{\nu}{x}\, J_{\nu}(x)-J_{\nu+1}(x),
\]
the orthogonality relation simplifies to the well known formula
\begin{equation}
  \sum_{k\in\mathbb{Z}\setminus\{0\}}j_{k,\nu}^{\,-2}\,
  R_{n,\nu+1}(j_{k,\nu})R_{m,\nu+1}(j_{k,\nu})
  = \frac{1}{2(n+\nu+1)}\,\delta_{mn},
  \label{eq:OGrel_Lommel}
\end{equation}
valid for $\nu>-1$ and $m,n\in\mathbb{Z}_{+}$.

\subsection{Lommel Polynomials in the variable $\nu$}

The Lommel polynomials can also be dealt with as polynomials in the
parameter $\nu$. Such polynomials are also orthogonal with the measure
of orthogonality supported by the zeros of a Bessel function of the
first kind regarded as a function of the order.

Let us consider a sequence of polynomials in the variable $\nu$ and
depending on a parameter $u\neq0$, $\{Q_{n}(u;\nu)\}_{n=0}^{\infty}$,
determined by the recurrence
\begin{equation}
  uQ_{n-1}(u;\nu)-nQ_{n}(u;\nu)+uQ_{n+1}(u;\nu)
  = \nu Q_{n}(u;\nu),\quad n\in\mathbb{Z}_{+},
\end{equation}
with the initial conditions $Q_{-1}(u;\nu)=0$, $Q_{0}(u;\nu)=1$.
According to (\ref{eq:OPs_P_F}),
\begin{equation}
  Q_{n}(u,\nu)
  = u^{-n}\,\frac{\Gamma(\nu+n)}{\Gamma(\nu)}\,\mathfrak{F}\!
  \left(\left\{ \frac{u}{\nu+k}\right\} _{k=0}^{n-1}\right)
  \ \ \text{for}\ n\in\mathbb{Z}_{+}.
\end{equation}
Comparing the last formula with (\ref{eq:lommel_F}) one observes that
\[
Q_{n}(u,\nu) = R_{n,\nu}(2u),\ \forall n\in\mathbb{Z}_{+}.
\]

The Bessel function $J_{\nu}(x)$ regarded as a function of $\nu$ has
infinitely many simple real zeros which are all isolated provided that
$x>0$, see \cite[Subsection 4.3]{StampachStovicek13a}. Below we denote
the zeros of $J_{\nu-1}(2u)$ by $\theta_{n}=\theta_{n}(u)$,
$n\in\mathbb{N}$, and restrict ourselves to the case $u>0$ since
$\theta_{n}(-u)=\theta_{n}(u)$.

The Jacobi matrix $J$ corresponding to this case, i.e. $J$ with the
diagonal $\lambda_{n}=-n$ and the weights $w_{n}=u$,
$n\in\mathbb{Z}_{+}$, is an unbounded self-adjoint operator with a
discrete spectrum (see \cite{StampachStovicek13a}). Hence the
orthogonality measure for $\{Q_{n}(u;\nu)\}$ has the form stated in
Remark~\ref{rem:isol_eigen_only}.  Thus, using (\ref{eq:jumps}) and
equation (\ref{eq:BesselJ_rel_F}) below, one arrives at the
orthogonality relation
\begin{equation}
  \sum_{k=1}^{\infty}\frac{J_{\theta_{k}}(2u)}
  {u\left(\partial_{z}\big|_{z=\theta_{k}}J_{z-1}(2u)\right)}\,
  R_{n,\theta_{k}}(2u)R_{m,\theta_{k}}(2u)
  = \delta_{mn},\ m,n\in\mathbb{Z}_{+}.
\end{equation}

Concerning the history, let us remark that initially this was
Dickinson in 1958 who proposed the problem of seeking a construction
of the measure of orthogonality for the Lommel polynomials in the
variable $\nu$ \cite{Dickinson58}. Ten years later, Maki described
such a construction in \cite{Maki}.

\section{A new class of orthogonal polynomials}

\subsection{Characteristic functions of the Jacobi matrices $J_{L}$ and $\tilde{J}_{L}$}

In this section we work with matrices $J_{L}$ and $\tilde{J_{L}}$
defined in (\ref{eq:Jacobi_mat_L}), (\ref{eq:lambda_w_coulomb}) and
(\ref{eq:Jacobi_mat_L_tilde}), (\ref{eq:lambda_w_coulomb_tilde}),
respectively. To have the weight sequence $w$ positive and the matrix
Hermitian, we assume, in the case of $J_{L}$, that $-1\neq L>-3/2$ if
$\eta\in\mathbb{R}\setminus\{0\}$, and $L>-3/2$ if $\eta=0$.
Similarly, in the case of $\tilde{J}_{L}$ we assume $L>-1/2$ and
$\eta\in\mathbb{R}$.

Recall that the regular and irregular Coulomb wave functions,
$F_{L}(\eta,\rho)$ and $G_{L}(\eta,\rho)$, are two linearly
independent solutions of the second-order differential equation
\begin{equation}
  \frac{d^{2}u}{d\rho^{2}}+\left[1-\frac{2\eta}{\rho}
    - \frac{L(L+1)}{\rho^{2}}\right]u=0,
  \label{eq:ODR_Coulomb}
\end{equation}
see, for instance, \cite[Chp. 14]{AbramowitzStegun}. One has the
Wronskian formula (see \cite[14.2.5]{AbramowitzStegun})
\begin{equation}
  F_{L-1}(\eta,\rho)G_{L}(\eta,\rho)
  - F_{L}(\eta,\rho)G_{L-1}(\eta,\rho)
  = \frac{L}{\sqrt{L^{2}+\eta^{2}}}\,.
  \label{eq:wronsk_F_G}
\end{equation}
Furthermore, the function $F_{L}(\eta,\rho)$ admits the decomposition
\cite[14.1.3 and 14.1.7]{AbramowitzStegun}
\begin{equation}
  F_{L}(\eta,\rho)
  = C_{L}(\eta)\rho^{L+1}\phi_{L}(\eta,\rho)
  \label{eq:F_L_decomp}
\end{equation}
where
\[
C_{L}(\eta) := \sqrt{\frac{2\pi\eta}{e^{2\pi\eta}-1}}\,
\frac{\sqrt{(1+\eta^{2})(4+\eta^{2})
    \dots(L^{2}+\eta^{2})}}{(2L+1)!!\, L!}
\]
and
\begin{equation}
  \phi_{L}(\eta,\rho) := e^{-i\rho}\,
  {}_{1}F_{1}(L+1-i\eta,2L+2,2i\rho).
  \label{eq:phi_rel_1F1}
\end{equation}
For $L$ not an integer, $C_{L}(\eta)$ is to be understood as
\begin{equation}
  C_{L}(\eta) = \frac{2^{L}e^{-\pi\eta/2}\,|\Gamma(L+1+i\eta)|}
  {\Gamma(2L+2)}\,.
  \label{eq:def_CL}
\end{equation}
In \cite{StampachStovicek13b}, the characteristic function for the
matrix $J_{L}$ has been derived. If expressed in terms of
$\mathcal{G}_{J_{L}}$, as defined in (\ref{eq:def_G_J}), the formula
simply reads
\begin{equation}
  \mathcal{G}_{J_{L}}(\rho)
  = \left(\prod_{k=L+1}^{\infty}(1-\lambda_{k}\rho)\right)
  \mathfrak{F}\!\left(\left\{ \frac{\gamma_{k}^{\,2}\rho}
      {1-\lambda_{k}\rho}\right\} _{k=L+1}^{\infty}\right)
  = \phi_{L}(\eta,\rho).
  \label{eq:G_J_L_eq_phi_L}
\end{equation}

For the particular values of parameters, $L=\nu-1/2$ and $\eta=0$, one
gets
\begin{equation}
  \mathfrak{F}\!
  \left(\left\{ \frac{\rho}{2(\nu+k)}\right\} _{k=1}^{\infty}\right)
 = \phi_{\nu-1/2}(0,\rho).
 \label{eq:char_fction_Coulomb2Bessel}
\end{equation}
It is also known that, see (\ref{eq:phi_rel_1F1}) and Eqs. 14.6.6 and
13.6.1 in \cite{AbramowitzStegun},
\begin{eqnarray}
  F_{\nu-1/2}(0,\rho) & = & \sqrt{\frac{\pi\rho}{2}}\, J_{\nu}(\rho),
  \label{eq:F_L_part}\\
  \phi_{\nu-1/2}(0,\rho) & = & e^{-i\rho}\,
  {}_{1}F_{1}(\nu+1/2,2\nu+1,2i\rho)
  \,=\,\Gamma(\nu+1)\left(\frac{2}{\rho}\right)^{\!\nu}
  J_{\nu}(\rho).
  \label{eq:phi_part}
\end{eqnarray}
(\ref{eq:char_fction_Coulomb2Bessel}) jointly with (\ref{eq:phi_part})
imply
\begin{equation}
  \mathfrak{F}\!\left(\left\{
      \frac{\rho}{\nu+k}\right\} _{k=1}^{\infty}\right)
  = \Gamma(\nu+1)\,\rho^{-\nu}J_{\nu}(2\rho).
  \label{eq:BesselJ_rel_F}
\end{equation}
The last formula has already been observed in
\cite{StampachStovicek11} and it holds for any $\nu\notin-\mathbb{N}$
and $\rho\in\mathbb{C}$.

Using the recurrence (\ref{eq:F_T_recur}) for $\mathfrak{F}$, one can
also obtain the characteristic function for $\tilde{J}_{L}$,
\begin{equation}
  \mathcal{F}_{\tilde{J}_{L}}(\rho^{-1})
  = \mathcal{F}_{J_{L}}(\rho^{-1})
  - \frac{\tilde{w}_{L}^{\,2}}{(\rho^{-1}
    - \tilde{\lambda}_{L})(\rho^{-1}-\lambda_{L+1})}\,
  \mathcal{F}_{J_{L+1}}(\rho^{-1}).
  \label{eq:F_J_tilde}
\end{equation}
Be reminded that $\phi_{L}(\eta,\rho)$ obeys the equations
\begin{eqnarray}
  \partial_{\rho}\phi_{L+1}(\eta,\rho)
  & = & \frac{2L+3}{\rho}\,\phi_{L}(\eta,\rho)
  - \left(\frac{2L+3}{\rho}+\frac{\eta}{L+1}\right)
  \phi_{L+1}(\eta,\rho),
  \label{eq:der_phi_L+1_eq_}\\
  \partial_{\rho}\phi_{L}(\eta,\rho)
  & = & \frac{\eta}{L+1}\,\phi_{L}(\eta,\rho)
  - \frac{\rho}{2L+3}\left(1+\frac{\eta^{2}}{(L+1)^{2}}\right)
  \phi_{L+1}(\eta,\rho),
  \label{eq:der_phi_L_eq_lc_phi_L_phi_L+1}
\end{eqnarray}
as it follows from \cite[14.2.1 and 14.2.2]{AbramowitzStegun}. A
straightforward computation based on (\ref{eq:G_J_L_eq_phi_L}),
(\ref{eq:F_J_tilde}) and (\ref{eq:der_phi_L_eq_lc_phi_L_phi_L+1})
yields
\begin{equation}
  \left(1+\frac{\eta\rho}{(L+1)^{2}}\right)
  \left(\prod_{n=L+1}^{\infty}\!
    \left(1+\frac{\eta\rho}{n(n+1)}\right)\right)
  \mathcal{F}_{\tilde{J}_{L}}(\rho^{-1})
  = \phi_{L}(\eta,\rho)+\frac{\rho}{L+1}\,
  \partial_{\rho}\phi_{L}(\eta,\rho).
\end{equation}
In view of (\ref{eq:F_L_decomp}), this can be rewritten as
\begin{equation}
  \mathcal{G}_{\tilde{J}_{L}}(\rho)
  = \phi_{L}(\eta,\rho)+\frac{\rho}{L+1}\,
  \partial_{\rho}\phi_{L}(\eta,\rho)
  = \frac{1}{(L+1)C_{L}(\eta)}\,\rho^{-L}\,
  \partial_{\rho}F_{L}(\eta,\rho).
  \label{eq:der_phi_L_eq_der_F_L}
\end{equation}

\subsection{Orthogonal polynomials associated with $F_{L}(\eta,\rho)$ \label{subsec:OPs_F_L}}

Following the general scheme outlined in Section~\ref{sec:OPs} (see
(\ref{eq:recur_OPs})) we denote by
\linebreak
$\{P_{n}^{(L)}(\eta;z)\}_{n=0}^{\infty}$ the sequence of OPs given by
the three-term recurrence
\begin{equation}
  zP_{n}^{(L)}(\eta;z)
  = w_{L+n}P_{n-1}^{(L)}(\eta;z)+\lambda_{L+n+1}
  P_{n}^{(L)}(\eta;z)+w_{L+n+1}P_{n+1}^{(L)}(\eta;z),
  \ \ n\in\mathbb{Z}_{+},
  \label{eq:PL_recur}
\end{equation}
with $P_{-1}^{(L)}(\eta;z)=0$ and $P_{0}^{(L)}(\eta;z)=1$. We again
restrict ourselves to the range of parameters $-1\neq L>-3/2$ if
$\eta\in\mathbb{R}\setminus\{0\}$, and $L>-3/2$ if $\eta=0$. Likewise
the Lommel polynomials, these polynomials are not included in Askey's
scheme \cite{Koekoek_etal}. Further let us denote
\begin{equation}
  R_{n}^{(L)}(\eta;\rho) := P_{n}^{(L)}(\eta;\rho^{-1})
  \label{eq:R_n_L_rel_P_n_L}
\end{equation}
for $\rho\neq0$, $n\in\mathbb{Z}_{+}$. According to
(\ref{eq:OPs_P_F}),
\begin{equation}
  P_{n}^{(L)}(\eta;z) = \left(\prod_{k=1}^{n}\,
    \frac{z-\lambda_{L+k}}{w_{L+k}}\right)
  \mathfrak{F}\!\left(\left\{ \frac{\gamma_{L+k}^{\,2}}
      {z-\lambda_{L+k}}\right\} _{k=1}^{n}\right)\!,
  \ \ n\in\mathbb{Z}_{+}.
  \label{eq:P_n_L_rel_F}
\end{equation}
Alternatively, these polynomials can be expressed in terms of the
Coulomb wave functions.

\begin{proposition}
  For $n\in\mathbb{Z}_{+}$and $\rho\neq0$ one has
  \begin{eqnarray*}
    R_{n}^{(L)}(\eta;\rho)
    & = & \frac{\sqrt{(L+1)^{2}+\eta^{2}}}{L+1}
    \sqrt{\frac{2L+2n+3}{2L+3}}\\
    \noalign{\smallskip}
    && \times \big(F_{L}(\eta,\rho)G_{L+n+1}(\eta,\rho)
    - F_{L+n+1}(\eta,\rho)G_{L}(\eta,\rho)\big).
  \end{eqnarray*}
\end{proposition}

\begin{proof}
  To verify this identity it suffices to check that the RHS fulfills
  the same recurrence relation as $R_{n}^{(L)}(\eta,\rho)$ does and
  with the same initial conditions. The RHS actually meets the first
  requirement as it follows from known recurrence relations for the
  Coulomb wave functions, see \cite[14.2.3]{AbramowitzStegun}.  The
  initial condition is a consequence of the Wronskian formula
  (\ref{eq:wronsk_F_G}).
\end{proof}

For the computations to follow it is useful to notice that the weights
$w_{n}$ and the normalization constants $C_{L}(\eta)$, as defined in
(\ref{eq:lambda_w_coulomb}) and (\ref{eq:def_CL}), respectively, are
related by the equation
\begin{equation}
  \prod_{k=0}^{n-1}w_{L+k}
  = \sqrt{\frac{2L+2n+1}{2L+1}}\,
  \frac{C_{L+n}(\eta)}{C_{L}(\eta)}\,,\ \ n=0,1,2,\ldots.
  \label{eq:prod_wk_eq_CL}
\end{equation}

\begin{proposition}
  For the above indicated range of parameters and $\rho\neq0$,
  \begin{equation}
    \lim_{n\rightarrow\infty}\sqrt{(2L+3)(2L+2n+1)}\,
    C_{L+n}(\eta)\rho^{L+n}R_{n-1}^{(L)}(\eta;\rho)
    = \sqrt{1+\frac{\eta^{2}}{(L+1)^{2}}}\, F_{L}(\eta,\rho).
    \label{eq:lim_RLn}
  \end{equation}
\end{proposition}

\begin{proof}
  Referring to (\ref{eq:R_n_L_rel_P_n_L}) and (\ref{eq:P_n_L_rel_F}),
  $R_{n}^{(L)}(\eta;\rho)$ can be expressed in terms of the function
  $\mathfrak{F}$. The sequence whose truncation stands in the argument
  of $\mathfrak{F}$ on the RHS of (\ref{eq:P_n_L_rel_F}) belongs to
  the domain $D$ defined in (\ref{eq:def_D}) and so one can apply the
  second equation in (\ref{eq:lim_F_T_n}). Concerning the remaining
  terms occurring on the LHS of (\ref{eq:lim_RLn}), one readily
  computes, with the aid of (\ref{eq:prod_wk_eq_CL}), that
  \begin{eqnarray*}
    &  & \lim_{n\rightarrow\infty}\,\sqrt{(2L+3)(2L+2n+1)}\,
    C_{L+n}(\eta)\rho^{L+n}\,\left(\prod_{k=1}^{n-1}\,
      \frac{\rho^{-1}-\lambda_{k+L}}{w_{k+L}}\right)\\
    &  & =\,\frac{\sqrt{(L+1)^{2}+\eta^{2}}}{L+1}\,
    C_{L}(\eta)\rho^{L+1}\prod_{k=1}^{\infty}
    \left(1+\frac{\eta\rho}{(L+k)(L+k+1)}\right).
  \end{eqnarray*}
  Recalling (\ref{eq:G_J_L_eq_phi_L}) and (\ref{eq:F_L_decomp}), the
  result immediately follows.
\end{proof}

\begin{remark}
  Note that the polynomials $R_{n}^{(L)}(\eta;\rho)$ can be regarded
  as a generalization of the Lommel polynomials $R_{n,\nu}(x)$.
  Actually, if $\eta=0$ then the Jacobi matrix $J_{\nu-1/2}$ is
  determined by the sequences
  \[
  \lambda_{n}=0,\ \ w_{n} = 1/(2\sqrt{(\nu+n+1)(\nu+n+2)}\,).
  \]
  Thus the recurrence (\ref{eq:PL_recur}) reduces to
  (\ref{eq:Lommel_recur}) for $\eta=0$ and $L=\nu-1/2$. More
  precisely, one finds that $P_{n}^{(\nu-1/2)}(0;z)$ coincides with
  the polynomial $P_{n}(2z)$ from Subsection~\ref{subsec:Lommel_basic}
  for all $n$. In view of (\ref{eq:Lommel_R_eq_Pn}) this means
  \begin{equation}
    R_{n}^{(\nu-1/2)}(0;\rho) = \sqrt{\frac{\nu+n+1}{\nu+1}}\,
    R_{n,\nu+1}(\rho)\label{eq:Coulomb_rel_Lommel}
  \end{equation}
  for $n\in\mathbb{Z}_{+}$, $\rho\in\mathbb{C}\setminus\{0\}$ and
  $\nu>-1$.  In addition we remark that, for the same values of
  parameters, (\ref{eq:lim_RLn}) yields Hurwitz' limit formula (see
  \S9.65 in \cite{Watson})
  \[
  \lim_{n\rightarrow\infty}\; \frac{(\rho/2)^{\nu+n}}
  {\Gamma(\nu+n+1)}\, R_{n,\nu+1}(\rho) = J_{\nu}(\rho).
  \]
\end{remark}

A more explicit description of the polynomials $P_{n}^{(L)}(\eta;z)$
can be derived. Let us write
\begin{equation}
  P_{n}^{(L)}(\eta;z)=\sum_{k=0}^{n}c_{k}(n,L,\eta)z^{n-k}.
  \label{eq:genLommel_coeffs}
\end{equation}
\begin{proposition} \label{prop:genL_explicit} Let
  $\{Q_{k}(n,L;\eta);\, k\in\mathbb{Z}_{+}\}$ be a sequence of monic
  polynomials in the variable $\ensuremath{\eta}$ defined by the
  recurrence
  \begin{equation}
    Q_{k+1}(n,L;\eta) = \eta\, Q_{k}(n,L;\eta)-h_{k}(n,L)
    Q_{k-1}(n,L;\eta)\text{ }\ \text{for}\ k\in\mathbb{Z}_{+},
    \label{eq:Q_recurr_I}
  \end{equation}
  with the initial conditions $Q_{-1}(n,L;\eta)=0$,
  $Q_{0}(n,L;\eta)=1$, where
  \[
  h_{k}(n,L) =
  \frac{k(2L+k+1)(2n-k+2)(2L+2n-k+3)}{4(2n-2k+1)(2n-2k+3)}\,,
  \ \ensuremath{k\in\mathbb{Z}_{+}}.
  \]
  Then the coefficients $c_{k}(n,L,\eta)$ defined in
  (\ref{eq:genLommel_coeffs}) fulfill
  \begin{eqnarray}
    c_{k}(n,L,\eta) & \hskip-0.4em = & \hskip-0.4em \frac{\sqrt{2L+2n+3}}
    {(L+1)\sqrt{2L+3}}\left|\frac{\Gamma(L+2+i\eta)}
      {\Gamma(L+n+2+i\eta)}\right|\frac{\Gamma(2n-k+2)\,
      \Gamma(2L+2n-k+3)}{\Gamma(2n-2k+2)\,\Gamma(2L+k+2)}\nonumber \\
    \noalign{\smallskip} &  & \times\,\frac{2^{-n+k-1}}{k!}\,
    Q_{k}(n,L;\eta)
    \label{eq:coeff_ck_explicit}
  \end{eqnarray}
  for $k=0,1,2,\ldots,n$.
\end{proposition}

For the proof we shall need an auxiliary identity. Note that, if
convenient,
\linebreak
$Q_{k}(n,L;\eta)$ can be treated as a polynomial in $\eta$
with coefficients belonging to the field of rational functions in the
variables $n$, $L$.

\begin{lemma}
  The polynomials $Q_{k}(n,L;\eta)$ defined in
  Proposition~\ref{prop:genL_explicit} fulfill
  \begin{equation}
    Q_{k}(n,L;\eta)-\alpha_{k}(n,L)\, Q_{k}(n-1,L;\eta)
    -\beta_{k}(n,L)\,\eta\, Q_{k-1}(n-1,L;\eta) = 0
    \label{eq:Q_intermed_II}
  \end{equation}
  for $k=0,1,2,\ldots$, where
  \begin{eqnarray*}
    \alpha_{k}(n,L)
    & = & \frac{2(2n-2k+1)(L+n+1)}{(2n-k+1)(2L+2n-k+2)}\,, \\
    \beta_{k}(n,L)
    & = & \frac{k\,(2L+k+1)}{(2n-k+1)(2L+2n-k+2)}\,.
  \end{eqnarray*}
\end{lemma}

\begin{proof}
  It is elementary to verify that $\alpha_{k}(n,L)$, $\beta_{k}(n,L)$
  fulfill the identities
  \begin{eqnarray}
    &  & \alpha_{k}(n,L)+\beta_{k}(n,L)\,=\,1,
    \label{eq:aux_alpha_beta_i}\\
    &  & \alpha_{k+1}(n,L)h_{k}(n-1,L)-\alpha_{k-1}(n,L)h_{k}(n,L)
    \,=\,0,\label{eq:aux_alpha_h_ii}\\
    &  & \beta_{k}(n,L)h_{k-1}(n-1,L)-\beta_{k-1}(n,L)h_{k}(n,L)
    \,=\,0.\label{eq:aux_beta_h_iii}
  \end{eqnarray}
  Now to show (\ref{eq:Q_intermed_II}) one can proceed by induction in
  $k$. The case $k=0$ means $\alpha_{0}(n,L)=1$ which is obviously
  true. Furthermore, $Q_{1}(n,L;\eta)=\eta$ and so the case $k=1$ is a
  consequence of (\ref{eq:aux_alpha_beta_i}), with $k=1$. Suppose
  $k\geq2$ and that the identity is true for $k-1$ and $k-2$. Applying
  (\ref{eq:Q_recurr_I}) both to $Q_{k}(n,L;\eta)$ and
  $Q_{k}(n-1,L;\eta)$ and using (\ref{eq:aux_alpha_beta_i}),
  (\ref{eq:aux_alpha_h_ii}) one can show the LHS of
  (\ref{eq:Q_intermed_II}) to be equal to
  \begin{eqnarray*}
    && -\,h_{k-1}(n,L)\big(Q_{k-2}(n,L;\eta)-\alpha_{k-2}(n,L)
    Q_{k-2}(n-1,L;\eta)\big) \\
    && +\,\eta\big(Q_{k-1}(n,L;\eta)-Q_{k-1}(n-1,L;\eta)\big).
  \end{eqnarray*}
  Next we apply the induction hypothesis both to $Q_{k-1}(n,L;\eta)$
  and $Q_{k-2}(n,L;\eta)$ in this expression to find that it equals,
  up to a common factor $\eta$,
  \begin{eqnarray*}
    &  & -\beta_{k-1}(n,L)Q_{k-1}(n-1,L;\eta)
    +\beta_{k-1}(n,L)\,\eta\, Q_{k-2}(n-1,L;\eta)\\
    &  & -h_{k-1}(n,L)\beta_{k-2}(n,L)Q_{k-3}(n-1,L;\eta).
  \end{eqnarray*}
  Finally, making once more use of (\ref{eq:Q_recurr_I}), this time
  for the term $Q_{k-1}(n-1,L;\eta)$, one can prove the last
  expression to be equal to
  \[
  \big(h_{k-2}(n-1,L)\beta_{k-1}(n,L)
  - h_{k-1}(n,L)\beta_{k-2}(n,L)\big)Q_{k-3}(n-1,L;\eta) = 0
  \]
  as it follows from (\ref{eq:aux_beta_h_iii}).
\end{proof}

\begin{proof}[Proof of Proposition~\ref{prop:genL_explicit}]
  Write the polynomials $P_{n}^{(L)}(\eta;z)$ in the form
  (\ref{eq:genLommel_coeffs}) and substitute the RHS of
  (\ref{eq:coeff_ck_explicit}) for the coefficients $c_{k}(n,L,\eta)$.
  Plugging the result into (\ref{eq:PL_recur}) one finds that the
  recurrence relation for the sequence $\{P_{n}^{(L)}(\eta;z)\}$ is
  satisfied if and only if the terms $Q_{k}(n,L;\eta)$ from the
  substitution fulfill
  \begin{eqnarray}
    &  & a_{k}(n,L)Q_{k}(n,L;\eta)-b_{k}(n,L)
    Q_{k}(n-1,L;\eta)-c_{k}(n,L)\,\eta\, Q_{k-1}(n-1,L;\eta)\nonumber \\
    &  & +\, d_{k}(n,L,\eta)Q_{k-2}(n-2,L;\eta) \,=\, 0
    \label{eq:2nd_recurr_Q_III}
  \end{eqnarray}
  for $k,n\in\mathbb{N}$, $n\geq k$, where we again put
  $Q_{-1}(n,L;\eta)=0$, $Q_{0}(n,L;\eta)=1$, and
  \begin{eqnarray*}
    &  & a_{k}(n,L)=\frac{(2n-k)(2n-k+1)(2L+2n-k+1)(2L+2n-k+2)}
    {(2L+k)(2L+k+1)(L+n+1)}\,,\\
    &  & b_{k}(n,L)=\frac{4(2L+2n+1)(n-k)(2n-2k+1)}{(2L+k)(2L+k+1)}\,,\\
    &  & c_{k}(n,L)=\frac{k(2L+2n+1)(2n-k)(2L+2n-k+1)}
    {(2L+k)(L+n)(L+n+1)}\,,\\
    &  & d_{k}(n,L,\eta)
    = \frac{(k-1)k\left(\eta^{2}+(L+n)^{2}\right)}{L+n}\,.
  \end{eqnarray*}
  Notice that for $k=0$ one has
  \begin{eqnarray*}
    &  & a_{0}(n,L)Q_{0}(n,L,\eta)-b_{0}(n,L)Q_{0}(n-1,L,\eta)
    -c_{0}(n,L)\,\eta\, Q_{-1}(n-1,L,\eta)\\
    &  & =\, a_{0}(n,L)-b_{0}(n,L)\,=\,0.
  \end{eqnarray*}

  One also observes, as it should be, that the relation
  (\ref{eq:2nd_recurr_Q_III}) determines the terms $Q_{k}(n,L;\eta)$
  unambiguously. In fact, to see it one can proceed by induction in
  $k$. Suppose $k>0$ and all terms $Q_{j}(n,L;\eta)$ are already known
  for $j<k$, $n\geq j$. Putting $n=k$ in (\ref{eq:2nd_recurr_Q_III})
  one can express $Q_{k}(k,L;\eta)$ in terms of $Q_{k-1}(k-1,L;\eta)$
  and
  \linebreak
  $Q_{k-2}(k-2,L;\eta)$ since $b_{k}(k,L)=0$. Then, treating $k$ as
  being fixed and $n$ as a variable, one can interpret
  (\ref{eq:2nd_recurr_Q_III}) as a first order difference equation in
  the index $n$, with a right hand side, for an unknown sequence
  $\left\{ Q_{k}(n,L;\eta);\, n\geq k\right\} $.  The initial
  condition for $n=k$ is now known as well as the right hand side and
  so the difference equation is unambiguously solvable.

  To prove the proposition it suffices to verify that if
  $\{Q_{k}(n,L;\eta)\}$ is a sequence of monic polynomials in the
  variable $\eta$ defined by the recurrence (\ref{eq:Q_recurr_I}) then
  it obeys, too, the relation (\ref{eq:2nd_recurr_Q_III}). To this
  end, one may apply repeatedly the rule (\ref{eq:Q_intermed_II}) to
  bring the LHS of (\ref{eq:2nd_recurr_Q_III}) to the form
  \begin{equation}
    e_{0}(n,L)Q_{k}(n-2,L;\eta)+e_{1}(n,L)\,\eta\,
    Q_{k-1}(n-2,L;\eta)+e_{2}(n,L,\eta)Q_{k-2}(n-2,L;\eta)
    \label{eq:aux_LHS_2nd_Q}
  \end{equation}
  where
  \begin{eqnarray*}
    e_{0}(n,L) & = & a_{k}(n,L)\alpha_{k}(n-1,L)\alpha_{k}(n,L)
    - b_{k}(n,L)\alpha_{k}(n-1,L),\\
    e_{1}(n,L) & = & a_{k}(n,L)\alpha_{k-1}(n-1,L)\beta_{k}(n,L)
    +a_{k}(n,L)\alpha_{k}(n,L)\beta_{k}(n-1,L)\\
    &  & -\, b_{k}(n,L)\beta_{k}(n-1,L)-c_{k}(n,L)\alpha_{k-1}(n-1,L),
  \end{eqnarray*}
  and
  \[
  e_{2}(n,L,\eta) = \big(a_{k}(n,L)\beta_{k-1}(n-1,L)\beta_{k}(n,L)
  - c_{k}(n,L)\beta_{k-1}(n-1,L)\big)\eta^{2}+d_{k}(n,L,\eta).
  \]
  Direct evaluation then gives
  \[
  e_{1}(n,L)/e_{0}(n,L) = -1,\ e_{2}(n,L,\eta)/e_{0}(n,L)
  = h_{k-1}(n-2,L).
  \]
  Referring to the defining relation (\ref{eq:Q_recurr_I}), this
  proves (\ref{eq:aux_LHS_2nd_Q}) to be equal to zero indeed.
\end{proof}

\begin{remark}
  Let us shortly discuss what Proposition~\ref{prop:genL_explicit}
  tells us in the particular case when $\eta=0$, $L=\nu-1/2$. The
  recurrence (\ref{eq:Q_recurr_I}) is easily solvable for $\eta=0$.
  One has $Q_{2k+1}(n,L;0)=0$ and
  \[
  Q_{2k}(n,L;0) = (-1)^{k}\,\frac{(2k)!\,(n-k)!\,(2n-4k+1)!}
  {k!\,(n-2k)!\,(2n-2k+1)!}\,\frac{\Gamma(L+k+1)\Gamma(L+n+2)}
  {\Gamma(L+1)\text{ }\Gamma(L+n-k+2)}
  \]
  for $k=0,1,2,\ldots$. Whence $c_{2k+1}(n,\nu-1/2,0)=0$ and
  \[
  c_{2k}\!\left(n,\nu-\frac{1}{2},0\right)
  = \sqrt{\frac{\nu+n+1}{\nu+1}}\,(-1)^{k}2^{n-2k}
  \binom{n-k}{k}\frac{\Gamma(\nu+n-k+1)}{\Gamma(\nu+k+1)}\,.
  \]
  Recalling (\ref{eq:Coulomb_rel_Lommel}) one rederives this way the
  explicit expression (\ref{eq:Lommel_explicit}) for the usual Lommel
  polynomials.
\end{remark}

Let us mention two more formulas. The first one is quite substantial
and shows that the polynomials $R_{n}^{(L)}(\eta,\rho)$ play the same
role for the Coulomb wave functions as the Lommel polynomials do for
the Bessel functions. It follows from the abstract identity
(\ref{eq:lincomb_Fx_FTx}) where we specialize $d=n$,
\begin{equation}
  x_{k} = \frac{\gamma_{L+k-1}^{\,2}}{\rho^{-1}-\lambda_{L+k-1}}\,,
  \label{eq:spec_xk}
\end{equation}
and again make use of (\ref{eq:prod_wk_eq_CL}). Thus we get
\begin{eqnarray}
  &  & R_{n}^{(L-1)}(\eta,\rho)F_{L}(\eta,\rho)
  -\frac{L+1}{L}\,\sqrt{\frac{2L+3}{2L+1}}\,
  \frac{\sqrt{\eta^{2}+L^{2}}}{\sqrt{\eta^{2}+(L+1)^{2}}}\,
  R_{n-1}^{(L)}(\eta,\rho)F_{L-1}(\eta,\rho)\nonumber \\
  &  & =\,\sqrt{\frac{2L+2n+1}{2L+1}}\, F_{L+n}(\eta,\rho),
  \label{eq:R_n_L_in_lincomb_F_L}
\end{eqnarray}
where $n\in\mathbb{Z}_{+}$, $0\neq L>-1/2$, $\eta\in\mathbb{R}$ and
$\rho\neq0$. Moreover, referring to (\ref{eq:F_L_part}) and
(\ref{eq:Coulomb_rel_Lommel}), one observes that relation
(\ref{eq:Lommel_in_lincomb_Bessel}) is a particular case of
(\ref{eq:R_n_L_in_lincomb_F_L}) if one lets $\eta=0$ and $L=\nu-1/2$.

Similarly, one can derive the announced second identity from
(\ref{eq:F_wronsk}) by making the same choice as that in
(\ref{eq:spec_xk}) but writing $z$ instead of $\rho^{-1}$. Recalling
(\ref{eq:P_n_L_rel_F}) one finds that
\begin{equation}
  P_{n}^{(L-1)}(\eta;z)P_{n+s}^{(L)}(\eta;z)
  - P_{n+s+1}^{(L-1)}(\eta;z)P_{n-1}^{(L)}(\eta;z)
  = \frac{w_{L}}{w_{L+n}}\, P_{s}^{(L+n)}(\eta;z)
  \label{eq:P_n_L_eq_wronskcomb}
\end{equation}
holds for all $n,s\in\mathbb{Z}_{+}$.

We conclude this subsection by describing the measure of orthogonality
for the generalized Lommel polynomials. To this end, we need an
auxiliary result concerning the zeros of the function
$\phi_{L}(\eta,\cdot)$ which is in fact a particular case of
Proposition~\ref{prop:G_J_simple_zeros} if we specialize the sequences
$w$ and $\lambda$ to the choice made in (\ref{eq:lambda_w_coulomb}).
From (\ref{eq:G_J_L_eq_phi_L}) we know that
$\phi_{L}(\eta,\rho)=\mathcal{G}_{J_{L}}(\rho)$ and we note that,
obviously, $J_{L+1}=J_{L}^{\,(1)}$. Thus we arrive at the following
statement.

\begin{proposition}\label{prop:phi_L_zeros}
  Let $-1\neq L>-3/2$ if $\eta\in\mathbb{R}\setminus\{0\}$, and
  $L>-3/2$ if $\eta=0$.  Then the zeros of the function
  $\phi_{L}(\eta,\cdot)$ form a countable subset of
  $\mathbb{R}\setminus\{0\}$ with no finite accumulation points.
  Moreover, the zeros of $\phi_{L}(\eta,.)$ are all simple, the
  functions $\phi_{L}(\eta,\cdot)$ and $\phi_{L+1}(\eta,\cdot)$ have
  no common zeros, and the zeros of the same sign of
  $\phi_{L}(\eta,\cdot)$ and $\phi_{L+1}(\eta,\cdot)$ mutually
  separate each other.
\end{proposition}

Let us arrange the zeros of $\phi_{L}(\eta,\cdot)$ into a sequence
$\rho_{L,n}$, $n\in\mathbb{N}$ (not indicating the dependence on
$\eta$ explicitly). According to Proposition~\ref{prop:phi_L_zeros} we
can do it so that
$0<|\rho_{L,1}|\leq|\rho_{L,2}|\leq|\rho_{L,3}|\leq\dots$.  Thus we
have
\begin{equation}
  \{\rho_{L,n};\, n\in\mathbb{N}\}=\{\rho\in\mathbb{R};\,
  \phi_{L}(\eta,\rho)=0\}
  = \{\rho\in\mathbb{R}\setminus\{0\};\, F_{L}(\eta,\rho)=0\}.
  \label{eq:zeros_phiL_ev_JL}
\end{equation}

\begin{proposition}\label{prop:OGrel_Coulomb}
  Let $-1\neq L>-3/2$ if $\eta\in\mathbb{R}\setminus\{0\}$, and
  $L>-3/2$ if $\eta=0$.  Then the orthogonality relation
  \begin{equation}
    \sum_{k=1}^{\infty}\rho_{L,k}^{\,-2}\,
    R_{n}^{(L)}(\eta;\rho_{L,k})R_{m}^{(L)}(\eta;\rho_{L,k})
    = \frac{(L+1)^{2}+\eta^{2}}{(2L+3)(L+1)^{2}}\,\delta_{mn}
    \label{eq:OGrel_Coulomb}
  \end{equation}
  holds for $m,n\in\mathbb{Z}_{+}$.
\end{proposition}

\begin{proof}
  By Proposition~\ref{prop:OGrel_Gfunc}, we have the orthogonality
  relation
  \[
  \int_{\mathbb{R}}P_{m}^{(L)}(\eta;\rho)P_{n}^{(L)}(\eta;\rho)\,
  \mbox{d}\mu(\rho) = \delta_{mn}
  \]
  where $\mbox{d}\mu$ is supported on the set
  $\{\rho_{L,n}^{\,-1};\,n\in\mathbb{N}\}\cup\{0\}$. Applying formula
  (\ref{eq:jump_mu_in_x}) combined with (\ref{eq:G_J_L_eq_phi_L}) and
  (\ref{eq:der_phi_L_eq_lc_phi_L_phi_L+1}) one finds that
  \[
  \mu(\rho_{L,k}^{\,-1})-\mu(\rho_{L,k}^{\,-1}-0)
  = -\rho_{L,k}^{\,-1}\,\frac{\phi_{L+1}(\eta,\rho_{L,k})}
  {\partial_{\rho}\phi_{L}(\eta,\rho_{L,k})}
  = \frac{(2L+3)(L+1)^{2}}{(L+1)^{2}+\eta^{2}}\,\rho_{L,k}^{\,-2}.
  \]
  We claim that $0$ is a point of continuity of $\mu$. Indeed, let us
  denote by $\Lambda_{k}$ the magnitude of the jump of $\mu$ at
  $\rho_{L,k}^{-1}$ if $k\in\mathbb{N}$, and at $0$ if $k=0$. Then,
  since $d\mu$ is a probability measure, one has
  \begin{equation}
    1 = \sum_{k=0}^{\infty}\Lambda_{k}
    = \Lambda_{0}+\frac{(2L+3)(L+1)^{2}}{(L+1)^{2}+\eta^{2}}
    \sum_{k=1}^{\infty}\rho_{L,k}^{\,-2}
    = \Lambda_{0}+\frac{(2L+3)(L+1)^{2}}{(L+1)^{2}+\eta^{2}}\,
    \|J_{L}\|_{2}^{\,2}
    \label{eq:tocomp_Lambda_0}
  \end{equation}
  where $\|J_{L}\|_{2}$ stands for the Hilbert-Schmidt norm of
  $J_{L}$.  This norm, however, can be computed directly,
  \[
  \|J_{L}\|_{2}^{\,2} = \sum_{n=1}^{\infty}
  \lambda_{L+n}^{\,2}+2\sum_{n=1}^{\infty}w_{L+n}^{\,2}
  = \frac{(L+1)^{2}+\eta^{2}}{(2L+3)(L+1)^{2}}\,.
  \]
  Comparing this equality to (\ref{eq:tocomp_Lambda_0}) one finds that
  $\Lambda_{0}=0$. To conclude the proof it suffices to recall
  (\ref{eq:R_n_L_rel_P_n_L}).
\end{proof}

In the course of the proof of Proposition~\ref{prop:OGrel_Coulomb} we
have shown that $0$ is a point of continuity of $\mu$. It follows that
$0$ is not an eigenvalue of the compact operator $J_{L}$.

\begin{corollary}
  Let $-1\neq L>-3/2$ if $\eta\in\mathbb{R}\setminus\{0\}$, and
  $L>-3/2$ if $\eta=0$. Then the operator $J_{L}$ is invertible.
\end{corollary}

\begin{remark}
  Again, letting $\eta=0$ and $L=\nu-1/2$ in (\ref{eq:OGrel_Coulomb})
  and recalling (\ref{eq:Coulomb_rel_Lommel}), one readily verifies
  that (\ref{eq:OGrel_Lommel}) is a particular case of
  (\ref{eq:OGrel_Coulomb}).
\end{remark}

\begin{remark}\label{rem:OPs_der_FL}
  In the same way as above one can define a sequence of OPs associated
  with the function $\partial_{\rho}F_{L}(\eta,\rho)$ or, more
  generally, with
  $\alpha{}F_{L}(\eta,\rho)+\rho\partial_{\rho}F_{L}(\eta,\rho)$ for
  some $\alpha$ real. But our results in this respect are not complete
  yet and are notably less elegant than those for the function
  $F_{L}(\eta,\rho)$, and so we confine ourselves to this short
  comment. Let us just consider the former particular case and call
  the corresponding sequence of OPs
  $\{\tilde{P}_{n}^{(L)}(\eta;z)\}_{n=0}^{\infty}$. It is defined by a
  recurrence analogous to (\ref{eq:PL_recur}) but now the coefficients
  in the relation are matrix entries of $\tilde{J}_{L}$ rather than
  those of $J_{L}$. The initial conditions are the same, and one has
  to restrict the range of parameters to the values $L>-1/2$ and
  $\eta\in\mathbb{R}$.  Let us denote
  $\tilde{R}_{n}^{(L)}(\eta;\rho):=\tilde{P}_{n}^{(L)}(\eta;\rho^{-1})$
  for $\rho\neq0$, $n\in\mathbb{Z}_{+}$. The zeros of the function
  $\rho\mapsto\partial_{\rho}F_{L}(\eta,\rho)$ can be shown to be all
  real and simple and to form a countable set with no finite
  accumulation points. One may arrange the zeros into a sequence. Let
  us call it $\{\tilde{\rho}_{L,n};\, n\in\mathbb{N}\}$.  Then the
  corresponding orthogonality measure $\mbox{d}\mu$ is again supported
  on the set
  $\{\tilde{\rho}_{L,k}^{\,-1};\,k\in\mathbb{N}\}\cup\{0\}$.  It is
  possible to directly compute the magnitude $\Lambda_{k}$ of the jump
  at $\tilde{\rho}_{L,k}^{\,-1}$ of the piece-wise constant function
  $\mu$ with the result
  \[
  \Lambda_{k} = \frac{L+1}{\tilde{\rho}_{L,k}^{\,2}
    -2\eta\tilde{\rho}_{L,k}-L(L+1)}\,.
  \]
  We propose that $\mu$ has no jump at the point $0$ or, equivalently,
  that the Jacobi matrix $\tilde{J}_{L}$ is invertible, but we have no
  proof for this hypothesis yet.
\end{remark}

\section{The spectral zeta function associated with $F_{L}(\eta,\rho)$ \label{subsec:zeta_F}}

Let us recall that the spectral zeta function of a positive definite
operator $A$ with a discrete spectrum whose inverse $A^{-1}$ belongs
to the $p$-th Schatten class is defined as
\[
\zeta^{(A)}(s) := \sum_{n=1}^{\infty}\frac{1}{\lambda_{n}^{\, s}}
= \Tr A^{-s},\quad\Re s\geq p,
\]
where $0<\lambda_{1}\leq\lambda_{2}\leq\lambda_{3}\leq\dots$ are the
eigenvalues of $A$. The zeta function can be used to approximately
compute the ground state energy of $A$, i.e. the lowest eigenvalue
$\lambda_{1}$. This approach is known as Euler's method (initially
applied to the first positive zero of the Bessel function $J_{0}$)
which is based on the inequalities
\begin{equation}
  \zeta^{(A)}(s)^{-1/s}<\lambda_{1}
  < \frac{\zeta^{(A)}(s)}{\zeta^{(A)}(s+1)}\,,\quad s\geq p.
  \label{eq:zeta_ineq}
\end{equation}
In fact, the inequalities in (\ref{eq:zeta_ineq}) become equalities in
the limit $s\rightarrow\infty$.

In this section we describe recursive rules for the zeta function
associated with the regular Coulomb wave function. The procedure can
be applied, however, to a wider class of special functions. For
example, this approach can also be applied to the Bessel functions
resulting in the well known convolution formulas for the Rayleigh
function \cite{Kishore}.  Not surprisingly, the recurrences derived
below can be viewed as a generalization of these known results.

Recall also that the regularized determinant,
\[
\det\!{}_{2}(1+A) := \det\left((1+A)\exp(-A)\right),
\]
is well defined if $A$ is a Hilbert-Schmidt operator, i.e. belonging
to the second Schatten class, on a separable Hilbert space. Moreover,
the regularized determinant is continuous in the Hilbert-Schmidt norm
\cite[Theorem~9.2.]{Simon}.

Referring to (\ref{eq:Jacobi_mat_L}), (\ref{eq:G_J_L_eq_phi_L}) and
(\ref{eq:zeros_phiL_ev_JL}), we start from the identity
\begin{eqnarray*}
  &  & \det\!{}_{2}(1-\rho J_{L,n})
  = \exp(\rho\Tr J_{L,n})\det(1-\rho J_{L,n})\\
  &  & = \exp\!\left(\frac{\eta\rho}{L+n+1}
    - \frac{\eta\rho}{L+1}\right)\left(\prod_{k=1}^{n}
    (1-\rho\lambda_{k+L})\right)
  \mathfrak{F}\!\left(\left\{ \frac{\gamma_{L+k}^{\,2}}
      {\lambda_{L+k}-\rho^{-1}}\right\} _{k=1}^{n}\right)
\end{eqnarray*}
where $J_{L,n}$ stands for the $n\times n$ truncation of $J_{L}$. The
formula can be verified straightforwardly by induction in $n$ with the
aid of the rule (\ref{eq:F_T_recur}). Now, sending $n$ to infinity and
using (\ref{eq:G_J_L_eq_phi_L}) we get
\begin{equation}
  \det\!{}_{2}(1-\rho J_{L})
  = \exp\!\left(-\frac{\eta\rho}{L+1}\right)\phi_{L}(\eta,\rho).
  \label{eq:det_2_Coulomb}
\end{equation}
On the other hand, we have the Hadamard product formula
\begin{equation}
  \det\!{}_{2}(1-\rho J_{L}) = \prod_{n=1}^{\infty}
  \left(1-\frac{\rho}{\rho_{L,n}}\right)e^{\,\rho/\rho_{L,n}},
  \label{eq:Coulomb_prod}
\end{equation}
see \cite[Theorem~9.2.]{Simon}. Combining (\ref{eq:det_2_Coulomb}) and
(\ref{eq:Coulomb_prod}) one arrives at the Hadamard infinite product
expansion of $\phi_{L}(\eta,.)$,
\begin{equation}
  \phi_{L}(\eta,\rho) = e^{\eta\rho/(L+1)}\prod_{n=1}^{\infty}
  \left(1-\frac{\rho}{\rho_{L,n}}\right)e^{\,\rho/\rho_{L,n}}.
  \label{eq:phi_Hadam}
\end{equation}

Let us define
\[
\zeta_{L}(k) := \sum_{n=1}^{\infty}\frac{1}{\rho_{L,n}^{\, k}}\,,
\ \ k\geq2.
\]
In view of (\ref{eq:phi_Hadam}), we can expand the logarithm of
$\phi_{L}(\eta,\rho)$ into a power series,
\[
\ln\phi_{L}(\eta,\rho) = \frac{\eta\rho}{L+1}
-\sum_{n=1}^{\infty}\sum_{k=2}^{\infty}\frac{1}{k}
\left(\frac{\rho}{\rho_{L,n}}\right)^{\! k},
\]
whenever $\rho\in\mathbb{C}$, $|\rho|<|\rho_{L,1}|$. Whence
\[
\frac{\partial_{\rho}\phi_{L}(\eta,\rho)}{\phi_{L}(\eta,\rho)}
= \frac{\eta}{L+1}-\sum_{k=1}^{\infty}\zeta_{L}(k+1)\,\rho^{k}.
\]
Comparing this equality to (\ref{eq:der_phi_L_eq_lc_phi_L_phi_L+1})
one finds that
\begin{equation}
  \sum_{k=0}^{\infty}\zeta_{L}(k+2)\rho^{k}
  = \frac{1}{2L+3}\left(1+\frac{\eta^{2}}{(L+1)^{2}}\right)
  \frac{\phi_{L+1}(\eta,\rho)}{\phi_{L}(\eta,\rho)}
  \quad\mbox{ for }|\rho|<|\rho_{L,1}|.
  \label{eq:ratio_Coulomb_zeta}
\end{equation}
So one can obtain the values of $\zeta_{L}(k)$ for $k\in\mathbb{N}$,
$k\geq2$, by inspection of the Taylor series of the RHS in
(\ref{eq:ratio_Coulomb_zeta}).

However, an apparently more efficient tool to compute the values of
the zeta function would be a recurrence formula. To find it one has to
differentiate equation (\ref{eq:ratio_Coulomb_zeta}) with respect to
$\rho$ and to use both formulas (\ref{eq:der_phi_L+1_eq_}) and
(\ref{eq:der_phi_L_eq_lc_phi_L_phi_L+1}). This way one arrives at the
equation
\begin{eqnarray*}
  &  & \hskip-2em(2L+3)
  \left(1+\frac{\eta^{2}}{(L+1)^{2}}\right)^{\!-1}
  \sum_{k=1}^{\infty}k\,\zeta_{L}(k+2)\rho^{k}
  \,=\, (2L+3)\left(1-\frac{\phi_{L+1}(\eta,\rho)}
    {\phi_{L}(\eta,\rho)}\right)\\
  &  & \hskip8em-\,\frac{2\eta\rho}{L+1}
  \frac{\phi_{L+1}(\eta,\rho)}{\phi_{L}(\eta,\rho)}
  + \frac{\rho^{2}}{2L+3}\left(1+\frac{\eta^{2}}
    {(L+1)^{2}}\right)\left(\frac{\phi_{L+1}(\eta,\rho)}
    {\phi_{L}(\eta,\rho)}\right)^{\!2}\!.
\end{eqnarray*}
Using (\ref{eq:ratio_Coulomb_zeta}) to express
$\phi_{L+1}(\eta,\rho)/\phi_{L}(\eta,\rho)$ one obtains
\begin{eqnarray*}
  \sum_{k=1}^{\infty}k\,\zeta_{L}(k+2)\rho^{k}
  & = & 1+\frac{\eta^{2}}{(L+1)^{2}}-(2L+3)
  \sum_{k=0}^{\infty}\zeta_{L}(k+2)\rho^{k}\\
  &  & +\,\frac{2\eta}{L+1}\sum_{k=1}^{\infty}
  \zeta_{L}(k+1)\rho^{k}+\sum_{k=0}^{\infty}
  \sum_{l=0}^{k}\zeta_{L}(l+2)\zeta_{L}(k-l+2)\rho^{k+2}.
\end{eqnarray*}
Now it suffices to equate coefficients at the same powers of $\rho$.
In particular, for the absolute term we get
\begin{equation}
  \zeta_{L}(2) = \frac{1}{2L+3}
  \left(1+\frac{\eta^{2}}{(L+1)^{2}}\right)\!.
  \label{eq:zeta_F_2}
\end{equation}
Note that $\zeta_{L}(2)$ is the square of the Hilbert-Schmidt norm of
$J_{L}$. The desired recurrence relation reads
\begin{equation}
  \zeta_{L}(k+1) = \frac{1}{2L+k+2}
  \left(\frac{2\eta}{L+1}\,\zeta_{L}(k)
    +\sum_{l=1}^{k-2}\zeta_{L}(l+1)\zeta_{L}(k-l)\right)\!,
  \ \ k=2,3,4,\dots.
  \label{eq:zeta_F_genrecur}
\end{equation}

As described above, bounds on the first (in modulus) zero $\rho_{L,1}$
can be determined with the aid of the zeta function. The operator
$J_{L}$ is not positive, however, and so the bounds should be written
as follows
\[
\zeta_{L}(2s)^{-1/s}<\rho_{L,1}^{\,2}
< \frac{\zeta_{L}(2s)}{\zeta_{L}(2s+2)}\,,\ \ s\geq1.
\]
In the simplest case, for $s=1$, we get the estimates
\[
\frac{(2L+3)(L+1)^{2}}{(L+1)^{2}+\eta^{2}}
< \rho_{L.1}^{\,2}<\frac{(2L+3)(2L+5)(L+2)(L+1)^{2}}
{(L+4)\eta^{2}+(L+2)(L+1)^{2}}\,.
\]

Further let us examine the particular case when $\eta=0$ and
$L=\nu-1/2$.  Then the rules (\ref{eq:zeta_F_2}) and
(\ref{eq:zeta_F_genrecur}) reproduce the well known recurrence
relations for the Rayleigh function $\sigma_{2n}(\nu)$, with $n\geq2$
and $\nu>-1$ \cite{Kishore}.  Recall that
\[
\sigma_{2n}(\nu) := \sum_{k=1}^{\infty}j_{\nu,k}^{\,-2n}
\]
where $j_{\nu,k}$ denotes the $k$-th positive zero of the Bessel
function $J_{\nu}$. The recurrence relation reads
\[
\sigma_{2}(\nu) = \frac{1}{4(\nu+1)}\,,\quad\sigma_{2n}(\nu)
= \frac{1}{n+\nu}\sum_{k=1}^{n-1}\sigma_{2k}(\nu)\sigma_{2n-2k}(\nu)
\ \ \text{for}\ n=2,3,4,\ldots.
\]

\begin{remark}
  Let us remark that instead of (\ref{eq:zeta_F_genrecur}) one can
  derive a recurrence relation in a form which is a linear combination
  of zeta functions. Rewrite equation (\ref{eq:ratio_Coulomb_zeta}) as
  \[
  \left(1+\frac{\eta^{2}}{(L+1)^{2}}\right)
  \phi_{L+1}(\eta,\rho)=(2L+3)\phi_{L}(\eta,\rho)
  \sum_{k=0}^{\infty}\zeta_{L}(k+2)\rho^{k}
  \]
  and replace everywhere the function $\phi_{L}$ by the power
  expansion
  \[
  \phi_{L}(\eta,\rho) = e^{-i\rho}\sum_{k=0}^{\infty}
  \frac{(L+1-i\eta)_{k}}{(2L+2)_{k}}\frac{(2i\rho)^{k}}{k!}\,.
  \]
  After obvious cancellations and equating coefficients at the same
  powers of $\rho$ on the both sides one arrives at the identity
  \[
  \frac{2[(L+1)^{2}+\eta^{2}]}{(L+1)(L+1-i\eta)}
  \frac{\Gamma(L+2-i\eta+k)}{\Gamma(2L+4+k)\, k!}
  = \sum_{l=0}^{k}\frac{\Gamma(L+1-i\eta+k-l)(2i)^{-l}}
  {\Gamma(2L+2+k-l)(k-l)!}\,\zeta_{L}(l+2),
  \]
  which holds for any $k\in\mathbb{Z}_{+}$, $L>-1$ and
  $\eta\in\mathbb{R}$.
\end{remark}

\begin{remark}
  The orthogonality measure $\mbox{d}\mu$ for the sequence of OPs
  $\{P_{n}^{(L)}(\eta,\rho)\}$, as described in
  Proposition~\ref{prop:OGrel_Coulomb}, fulfills
  \[
  \int_{\mathbb{R}}f(x)\,\mbox{d}\mu(x)
  = \frac{(2L+3)(L+1)^{2}}{(L+1)^{2}+\eta^{2}}
  \sum_{k=1}^{\infty}\rho_{L,k}^{\,-2}\, f(\rho_{L,k}^{\,-1})
  \]
  for every $f\in C(\mathbb{R})$. Consequently, the moment sequence
  associated with the measure $\mbox{d}\mu$ can be expressed in terms
  of the zeta function,
  \[
  m_{n} := \int_{\mathbb{R}}x^{n}\,\mbox{d}\mu(x)
  = \frac{\zeta_{L}(n+2)}{\zeta_{L}(2)}\,,\ \ n\in\mathbb{Z}_{+}
  \]
  (recall also (\ref{eq:zeta_F_2})). In view of formulas
  (\ref{eq:zeta_F_2}) and (\ref{eq:zeta_F_genrecur}), this means that
  the moment sequence can be evaluated recursively.
\end{remark}

\begin{remark}
  This comment extends Remark~\ref{rem:OPs_der_FL}. We note that it is
  possible to derive formulas analogous to (\ref{eq:zeta_F_genrecur})
  for the spectral zeta function associated with the function
  $\partial_{\rho}F_{L}(\eta,\rho)$ though the resulting recurrence
  rule is notably more complicated in this case.  One may begin,
  similarly to (\ref{eq:Coulomb_prod}), with the identities
  \begin{eqnarray*}
    \det\!{}_{2}(1-\rho\tilde{J}_{L})
    & = & \prod_{n=1}^{\infty}\left(1-\frac{\rho}
      {\tilde{\rho}_{L,n}}\right)e^{\rho/\tilde{\rho}_{L,n}}\\
    & = & \exp\!\left(\!-\frac{(L+2)\eta\rho}{(L+1)^{2}}\right)\!
    \left(\phi_{L}(\eta,\rho)+\frac{\rho}{L+1}\,
      \partial_{\rho}\phi_{L}(\eta,\rho)\right)\!.
  \end{eqnarray*}
  Hence for
  $\psi_{L}(\eta,\rho):=\phi_{L}(\eta,\rho)+(\rho/(L+1))\,\partial_{\rho}\phi_{L}(\eta,\rho)$
  we have
  \begin{equation}
    \ln\psi_{L}(\eta,\rho) = \frac{(L+2)\eta\rho}{(L+1)^{2}}
    -\sum_{n=1}^{\infty}\sum_{k=2}^{\infty}\frac{1}{k}
    \left(\frac{\rho}{\tilde{\rho}_{L,n}}\right)^{\! k}
    \label{eq:ln_psi_L}
  \end{equation}
  whenever $\rho\in\mathbb{C}$, $|\rho|<|\tilde{\rho}_{L,1}|$. Let us
  define
  \[
  \tilde{\zeta}_{L}(k) := \sum_{n=1}^{\infty}
  \frac{1}{\tilde{\rho}_{L,n}^{\, k}}\,,\ \ k\geq2.
  \]
  Now one can apply manipulations quite similar to those used in case
  of the zeta function associated with $F_{L}(\eta,\rho)$.
  Differentiating equation (\ref{eq:ln_psi_L}) twice and always taking
  into account that $F_{L}(\eta,\rho)$ solves (\ref{eq:ODR_Coulomb})
  one arrives, after some tedious but straightforward computation, at
  the equation
  \begin{eqnarray*}
    &  & \frac{2(\rho-\eta)}{\rho^{2}-2\eta\rho-L(L+1)}
    \left[-L\rho-\frac{(L+2)\eta\rho^{2}}{(L+1)^{2}}
      +\sum_{k=2}^{\infty}\tilde{\zeta}_{L}(k)\rho^{k+1}\right]\\
    &  & +\left[-L-\frac{(L+2)\eta\rho}{(L+1)^{2}}
      +\sum_{k=2}^{\infty}\tilde{\zeta}_{L}(k)\rho^{k}\right]^{2}\\
    &  & =\, L^{2}+\frac{2(L^{2}+L-1)\eta\rho}{(L+1)^{2}}
    -\rho^{2}+\sum_{k=2}^{\infty}(k+1)\tilde{\zeta}_{L}(k)\rho^{k}.
  \end{eqnarray*}
  From here the sought recurrence rules can be extracted in a routine
  way but we avoid writing them down explicitly because of their
  length and complexity.
\end{remark}

\section*{Acknowledgments}

The authors wish to acknowledge gratefully partial support from grant
No. GA13-11058S of the Czech Science Foundation.

\end{document}